\documentclass[oneside,english]{amsart}
\usepackage[T1]{fontenc}
\usepackage[latin9]{inputenc}
\usepackage{babel}
\usepackage{mathtools}
\usepackage{enumitem}
\usepackage{amstext}
\usepackage{amsthm}
\usepackage{amssymb}
\PassOptionsToPackage{normalem}{ulem}
\usepackage{ulem}
\usepackage[]
 {hyperref}

\makeatletter
\numberwithin{equation}{section}
\numberwithin{figure}{section}
\theoremstyle{plain}
\newtheorem{thm}{\protect\theoremname}[section]
\theoremstyle{remark}
\newtheorem{rem}[thm]{\protect\remarkname}
\theoremstyle{definition}
\newtheorem{example}[thm]{\protect\examplename}
\theoremstyle{definition}
\newtheorem{defn}[thm]{\protect\definitionname}
\theoremstyle{plain}
\newtheorem{lem}[thm]{\protect\lemmaname}
\theoremstyle{plain}
\newtheorem{prop}[thm]{\protect\propositionname}
\newlist{casenv}{enumerate}{4}
\setlist[casenv]{leftmargin=*,align=left,widest={iiii}}
\setlist[casenv,1]{label={{\itshape\ \casename} \arabic*.},ref=\arabic*}
\setlist[casenv,2]{label={{\itshape\ \casename} \roman*.},ref=\roman*}
\setlist[casenv,3]{label={{\itshape\ \casename\ \alph*.}},ref=\alph*}
\setlist[casenv,4]{label={{\itshape\ \casename} \arabic*.},ref=\arabic*}
\theoremstyle{plain}
\newtheorem{cor}[thm]{\protect\corollaryname}

\usepackage{amsthm}
\usepackage{etoolbox}
\usepackage{ifthen}
\usepackage{tikz}
\usetikzlibrary{calc}
\usepackage{tikz-cd}

\tikzcdset{%
    triple line/.code={\tikzset{%
        double distance = 3pt, 
        double=\pgfkeysvalueof{/tikz/commutative diagrams/background color}}},
    quadruple line/.code={\tikzset{%
        double distance = 5.3pt, 
        double=\pgfkeysvalueof{/tikz/commutative diagrams/background color}}},
    Rrightarrow/.code={\tikzcdset{triple line}\pgfsetarrows{tikzcd implies cap-tikzcd implies}},
    RRightarrow/.code={\tikzcdset{quadruple line}\pgfsetarrows{tikzcd implies cap-tikzcd implies}}
}

\makeatletter
\tikzset{
  column sep/.code=\def\pgfmatrixcolumnsep{\pgf@matrix@xscale*(#1)},
  row sep/.code   =\def\pgfmatrixrowsep{\pgf@matrix@yscale*(#1)},
  matrix xscale/.code=%
    \pgfmathsetmacro\pgf@matrix@xscale{\pgf@matrix@xscale*(#1)},
  matrix yscale/.code=%
    \pgfmathsetmacro\pgf@matrix@yscale{\pgf@matrix@yscale*(#1)},
  matrix scale/.style={/tikz/matrix xscale={#1},/tikz/matrix yscale={#1}}}
\def\pgf@matrix@xscale{1}
\def\pgf@matrix@yscale{1}
\makeatother

\theoremstyle{definition}
\newtheorem{convention}[thm]{Convention}
\newtheorem*{convention*}{Convention}
\newtheorem{notation}[thm]{Notation}
\newtheorem*{notation*}{Notation}

\makeatletter
\@namedef{subjclassname@2020}{%
  \textup{2020} Mathematics Subject Classification}
\makeatother

\numberwithin{thm}{subsection}

\usepackage{hyperref}

\makeatother

\providecommand{\casename}{Case}
\providecommand{\corollaryname}{Corollary}
\providecommand{\definitionname}{Definition}
\providecommand{\examplename}{Example}
\providecommand{\lemmaname}{Lemma}
\providecommand{\propositionname}{Proposition}
\providecommand{\remarkname}{Remark}
\providecommand{\theoremname}{Theorem}

\begin{document}
\global\long\def\Div{\operatorname{Div}}%

\global\long\def\ev{\operatorname{\textsc{ev}}}%

\global\long\def\EV{\operatorname{EV}}%

\global\long\def\NF{\operatorname{NF}}%

\global\long\def\nf{\operatorname{\textsc{nf}}}%

\global\long\def\Inv{\operatorname{Inv}}%

\global\long\def\Gar{\operatorname{Gar}}%

\global\long\def\im{\operatorname{im}}%

\global\long\def\End{\operatorname{End}}%

\global\long\def\sh{\operatorname{sh}}%

\global\long\def\h{\operatorname{h}}%

\global\long\def\t{\operatorname{t}}%

\global\long\def\op{\operatorname{op}}%

\global\long\def\sgn{\operatorname{sgn}}%

\global\long\def\mod{\operatorname{mod}}%

\title{Correspondence between factorability and normalisation in monoids}
\author{Alen \DJ uri\'{c}}
\address{Université Paris Cité, CNRS, Inria, IRIF, F-75013, Paris, France}
\email{alen.djuric@protonmail.com}
\thanks{%
\noindent\begin{minipage}[b][1\totalheight][t]{0.1\columnwidth}%
\begin{tikzpicture}[scale=0.027, every node/.style={scale=0.027}] \fill[fill={rgb,255:red,0;green,51;blue,153}] (-27,-18) rectangle (27,18); \pgfmathsetmacro\inr{tan(36)/cos(18)} \foreach \i in {0,1,...,11} { \begin{scope}[shift={(30*\i:12)}] \fill[fill={rgb,255:red,255;green,204;blue,0}] (90:2) \foreach \x in {0,1,...,4} { -- (90+72*\x:2) -- (126+72*\x:\inr) }; \end{scope} } \end{tikzpicture}%
\end{minipage}\hfill{}oo%
\begin{minipage}[b][1\totalheight][t]{0.8\columnwidth}%
This program has received funding from the European Union's Horizon
2020 research and innovation programme under the Marie Sk\l odowska-Curie
grant agreement No. 754362.%
\end{minipage}}
\date{30 December 2024}
\begin{abstract}
This article determines relations between two notions concerning monoids:
factorability structure, introduced to simplify the bar complex; and
quadratic normalisation, introduced to generalise quadratic rewriting
systems and normalisations arising from Garside families. Factorable
monoids are characterised in the axiomatic setting of quadratic normalisations.
Additionally, quadratic normalisations of class $\left(4,3\right)$
are characterised in terms of factorability structures and a condition
ensuring the termination of the associated rewriting system.  
\end{abstract}


\maketitle
\tableofcontents{}

\section*{Introduction}

This paper investigates combinatorial properties of a certain class
of monoids, seen from two different viewpoints, with a goal of unifying
the two. The main result answers the question, explicitly mentioned
in \cite{DDGKM} and \cite{DG}, of determining the relation between
these two approaches.

\subsection*{Factorability structures}

The notion of factorability structure on monoids and categories is
an extension by Wang \cite{Wan} and Hess \cite{Hes} of the definition
of factorability structure on groups introduced by Bödigheimer and
Visy \cite{BV} and Visy \cite{Vis}. Their motivation was to abstract
the structure, discovered in symmetric groups, that ensures the existence
of a normal form admitting some remarkable properties. In particular,
this normal form allows a reduction of the bar complex to a complex
having considerably fewer cells, well adapted for computing homology
of the algebraic structure in question. 

This reduction is achieved using the theory of collapsing schemes,
introduced in a topological flavour by Brown \cite{Bro2}, elaborated
for the algebraic setting by Cohen \cite{Coh1}, and rediscovered
(under the name of discrete Morse theory) by Forman \cite{For}. The
idea is to establish, for every nonnegative integer $n$, a bijection
(called collapsing scheme in \cite{Bro2}, and Morse matching in discrete
Morse theory) between a class of $n$-cells called redundant and a
class of $\left(n+1\right)$-cells called collapsible in such a way
that collapsing matched pairs preserves the homotopy type.

The idea of factorability for a given monoid $M$ and its generating
set $S$ is to determine a convenient way to split off a generator
from an element of the monoid. This is achieved by the notion of
factorability structure consisting of a factorisation map $\eta=\left(\eta',\overline{\eta}\right):M\to M^{2}$
subject to several axioms ensuring, in particular, a compatibility
with the multiplication in the monoid. In a manner of speaking, $\eta$
acts on an element $f$ of $M$ by splitting off a generator $\eta'\left(f\right)$.

For every factorability structure, there is an associated rewriting
system that is confluent but not necessarily terminating. However,
termination is obtained under the additional assumption that, for
all $s$ in $S$ and $f$ in $M$, the following equalities hold:
\begin{equation}
\eta'\left(sf\right)=\eta'\left(s\cdot\eta'\left(f\right)\right),\quad\overline{\eta}\left(sf\right)=\overline{\eta}\left(s\cdot\eta'\left(f\right)\right)\cdot\overline{\eta}\left(f\right).\label{eq:additional_assumption_intro}
\end{equation}

Section~\ref{sec:preliminaries} fixes basic terminology to be used
throughout the article. Basic notions and some results concerning
factorability are recollected in Section~\ref{sec:factorability}.

\subsection*{Quadratic normalisations}

Assume that a monoid is generated by a set $S$. By a normalisation,
we mean a syntactic transformation of an arbitrary word over $S$
into a \textquoteleft canonical\textquoteright{} one, called normal.
Quadratic normalisations in monoids, introduced by Dehornoy and Guiraud
\cite{DG} (influenced by Krammer \cite{Kra}), generalise, under
the same axiomatic setting, two well-known classes of normalisations:
those arising from quadratic rewriting systems, as studied in \cite{GGM}
for Artin-Tits monoids and in \cite{BCCL} and \cite{CGM} for plactic
monoids; and those arising from Garside families, as investigated
in \cite{DDGKM}, resulting from successive generalisations of the
greedy normal form in braid monoids, originating in the work of Garside
\cite{Gar}. Quadratic normalisations admit the following locality
properties: a word is normal if, and only if, its length-two factors
are normal; and the procedure of transforming a word into a normal
one consists of a finite sequence of rewriting steps, each of which
transforms a length-two factor.

The notion of class of a quadratic normalisation is defined in order
to measure the complexity of normalising length-three words. The class
$\left(m,n\right)$ means that every length-three word admits at most
$m$ rewriting steps starting from the left, and at most $n$ rewriting
steps starting from the right.

The class $\left(4,3\right)$ is explored in great detail in \cite{DG}
as it exhibits exceptionally favourable computational properties.
In particular, the rewriting system associated to a quadratic normalisation
of class $\left(4,3\right)$ is always convergent.

Quadratic normalisations are recalled in Section~\ref{sec:quadratic}.

\subsection*{Contributions}

We establish a correspondence between factorability structures and
quadratic normalisations for monoids, despite the different origins
and motivations for these two notions. By a correspondence, we mean
maps (in both directions) that are inverse to each other, up to technicalities,
between appropriate subclasses. Moreover, this bijection is compatible
with the associated rewriting systems.

Since the rewriting system associated with a factorability structure
is not necessarily terminating, whereas the rewriting system associated
to a quadratic normalisation of class $\left(4,3\right)$ is always
terminating, it has already been known that a quadratic normalisation
corresponding with a factorability structure is not necessarily of
class $\left(4,3\right)$.

It is shown here that a quadratic normalisation corresponding to a
factorability structure is always of class $\left(5,4\right)$ and
not smaller in general. A necessary and sufficient condition is given
for a quadratic normalisation to correspond to a factorability structure.
Thereby, we characterise factorable monoids in terms of quadratic
normalisations, thus adding another important family of monoids to
those unified under the axiomatic framework of quadratic normalisations.

In particular, a quadratic normalisation of class $\left(4,3\right)$
always yields a factorability structure, but not \emph{vice versa}.
However, the converse does hold under the stronger condition described
as follows. For a monoid $M$, consider the map $F\coloneqq\eta\circ\mu$
from the set $M^{2}$ to itself, with $\mu:M^{2}\to M$ denoting the
multiplication in $M$. Denote by $F_{1}$ (resp. $F_{2}$) the application
of $F$ to the first two elements (resp. the second and the third
element) of a triple in $M^{3}$. In general, a factorability structure
for $M$ and its generating set $S$ ensures the equality
\[
F_{1}F_{2}F_{1}F_{2}\left(r,s,t\right)=F_{2}F_{1}F_{2}\left(r,s,t\right)
\]
for each $S$-word $\left(r,s,t\right)$ such that $F_{2}F_{1}F_{2}\left(r,s,t\right)$
contains no $1$. The stronger condition states that this equality
holds for every $\left(r,s,t\right)$ in $S^{3}$. Since this condition
is implied by the class $\left(4,3\right)$, quadratic normalisations
of class $\left(4,3\right)$ are characterised in terms of factorability
structures.

Furthermore, we show that this stronger condition is equivalent to
the aforementioned additional assumption~(\ref{eq:additional_assumption_intro})
which is known to grant termination of the rewriting system associated
with a factorability structure. Simply put, the class $\left(4,3\right)$
equals factorability plus termination.

One of the benefits of the established correspondence between factorability
structures and quadratic normalisations is that it provides a way
of importing the results concerning homology, derived from the former
(see e.g. \cite{Vis}, \cite{Wan}, \cite{HO}) to the framework of
the latter, with the hope of generalising them to higher classes.
This would be our next step in the current direction of research.

To simplify the presentation, we are considering exclusively monoids,
but results stated here (recalled ones as well as new ones) mostly
extend to categories, seen as monoids with partial multiplication.
As a reminder that a monoid is thought of as a monoid of endomorphisms
(of an object), we tend to use letters $f$, $g$ and $h$ for elements
of a monoid.

\subsection*{Acknowledgments}

I warmly thank Viktoriya Ozornova for supervising this work and for
numerous remarks that highly improved the quality of the present text.
I am also grateful to the anonymous referee for meaningful suggestions.

\section{\protect\label{sec:preliminaries}Preliminaries}

\subsection{\protect\label{subsec:rewriting}Normal forms and rewriting systems}

If $S$ is a set, $S^{\ast}$ denotes the free monoid over $S$.
Elements of $S$ and $S^{\ast}$ are called $S$-letters and $S$-words,
respectively. So, an $S$-word is a finite sequence\footnote{In the broader context of categories, words are generalised to paths
of composable morphisms.} of $S$-letters, e.g.\ $\left(s_{1},\ldots,s_{n}\right)$. The prefix
$S$- is sometimes left out when the considered generating set is
evident from the context. The product in $S^{\ast}$ of two words
$u$ and $v$ is denoted by $u|v$. A letter $s$ is customarily
identified with the single-letter word $\left(s\right)$. Accordingly,
a word $\left(s_{1},\ldots,s_{n}\right)$ can be written as the product
(in $S^{\ast}$) of its letters: $s_{1}|\cdots|s_{n}$.

A monoid $M$ is said to be generated by a set $S$, often written
as $\left(M,S\right)$, if $M$ is a homomorphic image of the free
monoid $S^{\ast}$. Such a homomorphism is called an evaluation map
and denoted $\ev:S^{\ast}\to M$. A \textbf{normal form} for $M$
with respect to $S$ is a set-theoretic section, denoted by $\nf$,
of the evaluation map. To rephrase it, a normal form maps elements
of $M$ to distinguished representative words in $S^{\ast}$.

The length of $w$ in $S^{\ast}$ is denoted by $\left|w\right|$.
For an element $f$ of $M$, the minimal $S$-length of an $S$-word
representing $f$ is denoted by $\left|f\right|$. A normal form $\nf$
for a monoid $M$ with respect to a generating set $S$ is called
\textbf{geodesic} if, for every $f$ in $M$, the inequality $\left|\nf\left(f\right)\right|\leq\left|w\right|$
 holds for every $S$-word $w$ representing $f$, i.e. such that
$\ev\left(w\right)=f$. 

A (word) \textbf{rewriting system} is a pair $\left(S,R\right)$ consisting
of a set $S$ and a binary relation $R$ on $S^{\ast}$, whose elements
are called \textbf{rewriting rules}. An element $\left(u,v\right)$
of $R$ is also written as $u\rightarrow v$ to stress the fact that
it is directed. Seeing relations between words not as equalities
but as rewriting rules is a key concept of rewriting theory.

For a rewriting rule $\left(u,v\right)$, and for $w$ and $w'$ in
$S^{\ast}$, a pair $\left(w|u|w',w|v|w'\right)$ is called a \textbf{rewriting
step}. For $u$ and $v$ in $S^{\ast}$, we say that $u$ \textbf{rewrites}
to $v$ if there is a finite composable sequence of rewriting steps,
such that the source of the first step of the sequence is $u$ and
the target of the last step is $v$. A word $u$ is called \textbf{irreducible}
with respect to $R$ if there is no rewriting step whose source is
$u$. 

A rewriting system $\left(S,R\right)$ is called:
\begin{itemize}
\item \textbf{confluent} if any two rewriting sequences starting with the
same word can be completed in such a way that they eventually reach
a common result;
\item \textbf{normalising} if every $u$ in $S^{\ast}$ rewrites to at least
one irreducible word;
\item \textbf{terminating} if it admits no infinite rewriting sequence;
\item \textbf{convergent} if it is both confluent and terminating;
\item \textbf{reduced} (or minimal) if for every rewriting rule $u\rightarrow v$,
the word $v$ is irreducible with respect to $R$, and the word $u$
is irreducible with respect to $R\setminus\left\{ \left(u,v\right)\right\} $;
\item \textbf{strongly reduced} if it is reduced and, in addition, every
element of $S$ is irreducible;
\item \textbf{quadratic} if the source and target of every element of $R$
are of length $2$.
\end{itemize}
The monoid \textbf{presented by} a rewriting system $\left(S,R\right)$
is the quotient $M$ of the free monoid $S^{\ast}$ by the congruence
relation generated by $R$. If $\left(S,R\right)$ is confluent, the
irreducible word to which a word $w$ rewrites, if it exists, is denoted
by $\widehat{w}$. If $\left(S,R\right)$ is convergent, the map $M\to S^{\ast}$
defined by $f\mapsto\widehat{w}$, with $\ev\left(w\right)=f$, is
the normal form associated with the rewriting system $\left(S,R\right)$.
\begin{rem}
\label{rem:generating-(sub)set}If a generating set $S$ of a monoid
$M$ is a subset of $M$, then elements of $S$ can be regarded in
two ways: as length-one words in $S^{\ast}$, and as elements of $M$.
When a rewriting system presenting $M$ with respect to $S$ is strongly
reduced, this makes no essential difference, so elements of $S$ are
denoted in the same way, regardless of the viewpoint, relying on the
context to provide the proper interpretation. In particular, one can
say that a generating set contains (or that it does not contain) $1$,
the identity element of a monoid. This phrasing is the custom in the
context of factorability (see \cite{HO}, \cite{Hes}, \cite{Ozo}),
but not in the context of normalisation where a generating set is
commonly distinguished from its image under the evaluation map (see
\cite{DG}). So, we will emphasise such situation by calling $S$
a generating subset, not just a generating set, of $M$. When we characterise
factorable monoids in terms of quadratic normalisations (Subsection~\ref{subsec:characterisation}),
the corresponding normalisations will be eligible to share this custom
so there will be no need to emphasise it.
\end{rem}

For technical reasons, in the rest of this article, the letter $S$
will be reserved for the following purpose.

\begin{convention}\label{con:single-letter-for-1}When the letter
$S$ is used to denote a generating set of a monoid, it is understood
that $S$ is a pointed set\footnote{That is a set equipped with a distinguished element, called basepoint,
enjoying a special treatment in the given context.} whose basepoint is a letter, customarily denoted by $e$, representing
the identity of the considered monoid. In accordance with Remark~\ref{rem:generating-(sub)set},
the basepoint of $S$ is denoted by $1$ if $S$ is a generating subset
of the considered monoid.

On the other hand, if we exclude this letter from $S$, we write $S_{+}$
for the resulting generating set.\end{convention}

\subsection{Divisibility in monoids}

A monoid $M$ is said to be \textbf{left-cancellative} (resp. \textbf{right-cancellative})
if, for every $f$, $g$ and $g'$ in $M$, the equality $fg=fg'$
(resp. $gf=g'f$) implies the equality $g=g'$.

An element $f$ of a monoid $M$ is called a \textbf{left divisor}
of $g$ in $M$, and $g$ is called a \textbf{right multiple} of $f$,
denoted by $f\preceq g$, if there is an element $f'$ in $M$ such
that $ff'=g$. If $M$ is left-cancellative, then the element $f'$
is uniquely determined and called \textbf{the right complement of
$f$ in $g$}.

\section{\protect\label{sec:factorability}Factorability structures}

This section recalls the notion of factorability structure. Subsection~\ref{subsec:notion-of-factorability}
recollects the basic terminology. In Subsection~\ref{subsec:local-factorability},
we recall an alternative approach to factorability through the notion
of local factorability. Certain notions are redefined here in order
to overcome the issues arising from the original definition, which
are pointed out in Subsection~\ref{subsec:deviation}. Finally, Subsection~\ref{subsec:factorability-rewriting}
recalls the rewriting system associated to a factorability structure.
For elaboration, the reader is referred to \cite{HO} and \cite{Ozo}.

\subsection{\protect\label{subsec:notion-of-factorability}Factorability structures}

\begin{convention}\label{con:flip}Let us adopt the convention that
elements of a finite sequence are indexed starting from the leftmost
one, as in $\left(s_{1},s_{2},\ldots,s_{n}\right)$, thereby not following
the convention used in \cite{HO} where elements are indexed starting
from the rightmost one. The purpose is to make the notions that concern
factorability more easily comparable (in Section~\ref{sec:contribution})
with those concerning normalisation. \end{convention}

A pair $\left(f,g\right)$ in $M^{2}$ is called \textbf{geodesic}
if $\left|fg\right|=\left|f\right|+\left|g\right|$.

Let $M$ be a monoid, and let $S$ be a generating subset of $M$.
A \textbf{factorisation map} for $\left(M,S\right)$ is a map $\eta=\left(\eta',\overline{\eta}\right):M\to M^{2}$
satisfying the following conditions:
\begin{itemize}
\item for $f$ in $M\setminus\left\{ 1\right\} $, the element $\eta'\left(f\right)$
in $S_{+}$ is a left divisor of $f$, and the element $\overline{\eta}\left(f\right)$
is a right complement of $\eta'\left(f\right)$ in $f$;
\item the pair $\left(\eta'\left(f\right),\overline{\eta}\left(f\right)\right)$
is geodesic;
\item $\eta$ maps $1$, the identity element of $M$, to $\eta\left(1\right)=\left(1,1\right)$.
\end{itemize}
Whenever confusion is unlikely, $\eta'\left(f\right)$ and $\overline{\eta}\left(f\right)$
are abbreviated to $f'$ and $\overline{f}$, respectively.
\begin{example}
Assume that $M$ is a free commutative monoid generated by a nonempty
finite totally ordered set. Define $\eta=\left(\eta',\overline{\eta}\right):M\to M^{2}$
by setting $\eta'\left(f\right)$ to be the least left divisor of
$f$, lying in the generating set. Note that this is well-defined
since the left cancellation property of $M$ implies uniqueness of
right complements, so knowing $\eta'\left(f\right)$ determines $\overline{\eta}\left(f\right)$.
\end{example}

\begin{notation}\label{not:at-positions}Let $A$ be a set, and let
$F$ be a map from $A^{k}$ to $A^{l}$. Then the (partial) map $F_{i}:A^{\ast}\to A^{\ast}$
consists of applying $F$ to $k$ consecutive elements starting from
position $i$, i.e.\ to the elements at positions $i,i+1,\ldots,i+k-1$.\end{notation}
\begin{example}
\label{exa:illustration-partial-application}For the sake of illustration,
take the set $A=\left\{ a,b,c\right\} $ totally ordered by $a<b<c$.
 We write $<^{\ast}$ for the lexicographic extension of $<$ to
$A^{\ast}.$ Let $F:A^{2}\to A^{2}$ map each length-two word $w$
to the $<^{\ast}$-minimal word obtained by simply permuting letters
of $w$ if needed. Then, we have:
\[
c|b|a\stackrel{F_{2}}{\mapsto}c|a|b\stackrel{F_{1}}{\mapsto}a|c|b\stackrel{F_{2}}{\mapsto}a|b|c.
\]
\end{example}

The multiplication in $M$ is denoted by $\mu:M^{2}\to M$, and $\mu\left(f,g\right)$
is often abbreviated to $f\cdot g$ or $fg$.
\begin{defn}[{\cite[Definition 2.1]{HO}}]
Let $M$ be a monoid, and let $S$ be a generating subset of $M$.
A \textbf{factorability structure} on $\left(M,S\right)$ is a factorisation
map $\eta:M\to M^{2}$ such that, denoting the map $\eta\circ\mu:M^{2}\to M^{2}$
by $F$, for every triple in $M^{3}$, the three maps
\[
F_{1}F_{2}F_{1}F_{2},\quad F_{2}F_{1}F_{2},\quad F_{2}F_{1}F_{2}F_{1}
\]
coincide or each map reduces the sum of the lengths of the elements
of the triple. If $\eta:M\to M^{2}$ is a factorability structure
on $\left(M,S\right)$, we call the triple $\left(M,S,\eta\right)$
a \textbf{factorable monoid}.
\end{defn}

Assume that $\left(M,S,\eta\right)$ is a factorable monoid. The
\textbf{normal form associated with the factorability structure $\eta$},
or the $\eta$-normal form, for short, is the map $\nf_{\eta}:M\to S^{\ast}$
defined as
\[
f\mapsto\eta_{\left|f\right|-1}\cdots\eta_{2}\eta_{1}\left(f\right).
\]

\begin{example}
The map $F:A^{2}\to A^{2}$ in Example~\ref{exa:illustration-partial-application}
can be regarded as a composition $\eta\circ\mu$ of the multiplication
in $A^{\ast}$ and a factorability structure splitting off the least
letter. For $f=bacabc$, we get
\[
\nf_{\eta}\left(f\right)=\eta_{5}\cdots\eta_{2}\eta_{1}\left(f\right)=\left(a,a,b,b,c,c\right).
\]
\end{example}

For a factorable monoid $\left(M,S,\eta\right)$, an $M$-word $x$
is said to be \textbf{stable at the $i$th position} if $F_{i}\left(x\right)=x$;
it is \textbf{everywhere stable} if it is stable at the $i$th position
for every $i$ in $\left\{ 1,\ldots,\left|x\right|-1\right\} $. The
normal form $\nf_{\eta}$ admits the following locality property. 
\begin{lem}[{\cite[Remark 2.1.27]{Hes}}]
\label{lem:normal-stable}Let $\left(M,S,\eta\right)$ be a factorable
monoid. Then, for every \textbf{$f$} in $M$, the $\eta$-normal
form is everywhere stable.
\end{lem}

Although it may appear that \cite[Remark 2.1.27]{Hes} uses an extra
condition, namely \textquoteleft the recognition principle\textquoteright ,
this is not the case since this condition is automatically satisfied
in a factorable monoid (see \cite[Lemma 2.2.2]{Hes}).

\subsection{\protect\label{subsec:local-factorability}Local factorability structure}

There is an alternative definition of factorability by use of the
notion of local factorability, due to Moritz Rodenhausen. In order
to resolve an issue detected in the original definition of local factorability
(to be addressed in the next subsection), we introduce the following
notation.

\begin{notation}\label{not:avoid-phi(s)}Let $A$ be a set, and let
$\varphi$ be a map from $A^{2}$ to itself. The composite map $\underrightarrow{\varphi}:A^{\ast}\to A^{\ast}$
is defined as
\[
w\mapsto\varphi_{\left|w\right|-1}\cdots\varphi_{1}\left(w\right).
\]
Note that, in the above composition, $\varphi_{\left|w\right|-1}$
is applied last, taking the rightmost length-two factor of $\varphi_{\left|w\right|-2}\cdots\varphi_{1}\left(w\right)$
as an argument. In particular, if $w$ has length $1$, then $\underrightarrow{\varphi}\left(w\right)=w.$
\end{notation}
\begin{defn}
\label{def:normalisation-map}Let $S$ be a pointed set with basepoint
$1$, and let $\varphi$ be a map from $S^{2}$ to itself. The \textbf{normalisation
map associated with the map $\varphi$} is a map \textbf{}$N_{\varphi}$
from $S^{\ast}$ to itself, defined as follows:
\begin{enumerate}
\item $N_{\varphi}$ of the empty word is the empty word;
\item \label{enu:N_phi-2}$N_{\varphi}$ of a word containing $1$ equals
$N_{\varphi}$ of the same word with $1$ removed;
\item \label{enu:N_phi-3}and
\[
N_{\varphi}\left(s_{1},\ldots,s_{n}\right)\coloneqq\begin{cases}
\underrightarrow{\varphi}\left(s_{1}|N_{\varphi}\left(s_{2},\ldots,s_{n}\right)\right) & \textrm{ if it contains no }1\\
N_{\varphi}\left(\underrightarrow{\varphi}\left(s_{1}|N_{\varphi}\left(s_{2},\ldots,s_{n}\right)\right)\right) & \textrm{ otherwise}.
\end{cases}
\]
\end{enumerate}
\end{defn}

\begin{rem}
Note that, by recursion, the computation of $N_{\varphi}\left(s_{1},\ldots,s_{n}\right)$
terminates and its length is bounded by $n$. Namely, all the recursive
calls for $N_{\varphi}$ are made on words of smaller length: directly
in (\ref{enu:N_phi-2}) and in the first case of (\ref{enu:N_phi-3});
and indirectly in the second case of (\ref{enu:N_phi-3}), which calls
for (\ref{enu:N_phi-2}). The length of $N_{\varphi}\left(s_{1},\ldots,s_{n}\right)$
is less than or equal to the length of $\underrightarrow{\varphi}\left(s_{1}|N_{\varphi}\left(s_{2},\ldots,s_{n}\right)\right)$,
which is less than or equal to the length of $N_{\varphi}\left(s_{2},\ldots,s_{n}\right)$
plus $1$. 
\end{rem}

We state anew the definition of local factorability structure on $\left(M,S\right)$,
using Definition~\ref{def:normalisation-map}.
\begin{defn}
\label{def:local-factorability-structure}Let $M$ be a monoid and
let $S$ be a generating subset of $M$. A \textbf{local factorability
structure} is a map $\varphi$ from $S^{2}$ to itself, having the
following properties:
\begin{enumerate}
\item \label{enu:local-factorability-structure-1}$M$ admits the presentation
\[
\left\langle S_{+}\vert\left\{ \left(s,t\right)=\varphi\left(s,t\right)\vert s,t\in S_{+}\right\} \right\rangle ;
\]
\item \label{enu:local-factorability-structure-2}$\varphi$ is idempotent;
\item \label{enu:local-factorability-structure-3}$\varphi\left(1,s\right)$
equals $\left(s,1\right)$ for every $s$ in $S_{+}$;
\item \label{enu:local-factorability-structure-4}for every $\left(r,s,t\right)$
in $S_{+}^{3}$, the equality
\[
\varphi_{1}\varphi_{2}\varphi_{1}\varphi_{2}\left(r,s,t\right)=\varphi_{2}\varphi_{1}\varphi_{2}\left(r,s,t\right)
\]
 holds or $\varphi_{2}\varphi_{1}\varphi_{2}\left(r,s,t\right)$ contains
$1$;
\item \textbf{\label{enu:local-factorability-structure-5}}the normalisation
map associated with $\varphi$ satisfies 
\[
N_{\varphi}\left(r,s,t\right)=N_{\varphi}\left(\varphi_{1}\left(r,s,t\right)\right)
\]
for every $\left(r,s,t\right)$ in $S^{3}$.
\end{enumerate}
\end{defn}

By Definition~\ref{def:normalisation-map}, for every $S$-word
$w$, the word $N_{\varphi}\left(w\right)$ contains no $1$. If we
add to $N_{\varphi}\left(w\right)$ a string of $1$'s on the right,
then the result is called an \textbf{extended form} of $N_{\varphi}\left(w\right)$.

\begin{lem}[{\cite[Lemma 2.3.5]{Ozo}}]
\label{lem:2121-stable}Let $S$ be a pointed set with basepoint
$1$, and let $\varphi$ be a map from $S^{2}$ to itself. If $\varphi$
formally satisfies the second, the third and the fourth condition
of Definition~\ref{def:local-factorability-structure}, then $\varphi_{1}\varphi_{2}\varphi_{1}\varphi_{2}\left(r,s,t\right)$
is an extended form of \textbf{$N_{\varphi}\left(r,s,t\right)$} for
every $S$-word $\left(r,s,t\right)$.
\end{lem}

A factorability structure is equivalent to a local factorability structure,
in the following sense.
\begin{thm}[{\cite[Theorem 3.4]{HO}}]
\label{thm:factorability-local-factorability}~
\begin{enumerate}
\item If $\left(M,S,\eta\right)$ is a factorable monoid, then the restriction
of the map $\eta\mu$ to $S^{2}$ defines a local factorability structure
on $M$.
\item Conversely, one can construct a factorability structure out of a local
factorability structure by setting $\eta\left(1\right)=\left(1,1\right)$
and 
\[
\eta\left(f\right)=\left(r_{1},\ev\left(r_{2},\ldots,r_{n}\right)\right)\textrm{ for }N_{\varphi}\left(w\right)=\left(r_{1},r_{2},\ldots,r_{n}\right),
\]
with $w$ being any $S$-word representing $f$.
\item These constructions are inverse to each other.
\item \label{enu:correspondence-eta-phi-4}By this correspondence, for $f$
in $M$, the $\eta$-normal form $\nf_{\eta}\left(f\right)$ equals
$N_{\varphi}\left(w\right)$ for any $S$-word $w$ representing $f$.
\end{enumerate}
\end{thm}

The proof of Theorem~\ref{thm:factorability-local-factorability}
can be found in \cite[Section 2.3]{Ozo}.
\begin{rem}
\label{rem:first-component-phi}Here are some observations about local
factorability structures that will be used implicitly from now on.
\begin{itemize}
\item The property~(\ref{enu:N_phi-3}) of Definition~\ref{def:normalisation-map}
implies $N_{\varphi}\left(s\right)=s$ for every $s$ in $S_{+}$.
\item For every $S$-word $\left(s,t\right)$, the first element of $\varphi\left(s,t\right)$
cannot be equal to $1$ unless the second element is equal to $1$.
Namely, assume the opposite: $\varphi\left(s,t\right)=\left(1,t'\right)$
for $t'\neq1$. Then the idempotency of $\varphi$ gives $\varphi\left(1,t'\right)=\left(1,t'\right)$,
which contradicts the property~(\ref{enu:local-factorability-structure-3})
of Definition~\ref{def:local-factorability-structure}. 
\item Note that, by Theorem~\ref{thm:factorability-local-factorability},
the equality $\varphi\left(s,t\right)=\left(1,1\right)$ holds if,
and only if, $st=1$ in $M$.
\end{itemize}
\end{rem}

\subsection{\protect\label{subsec:deviation}Deviation from the original definition}

As Convention~\ref{con:flip} hints, the original definition of
factorisation map, which separates a right divisor, is reformulated
in Subsection~\ref{subsec:notion-of-factorability} to separate a
left divisor, instead. The definition of local factorability structure
is also modified.

Let us recall the original definition of local factorability structure,
in order to justify its present modification (Definition~\ref{def:local-factorability-structure}).
For simplicity, we still assume Convention~\ref{con:single-letter-for-1},
so we do not actually copy the original\emph{ verbatim}, but we do
preserve its essence (as well as the convention of indexing from the
right).

Here is a recollection of \cite[Definition 3.3]{HO}. Let $M$ be
a monoid, and let $S$ be a generating subset of $M$. A local factorability
structure on $\left(M,S\right)$ is a map $\varphi$ from $S^{2}$
to itself, having the following properties:
\begin{enumerate}
\item $M$ admits the presentation
\[
\left\langle S_{+}\vert\left\{ \left(t,s\right)=\varphi\left(t,s\right)\vert s,t\in S_{+}\right\} \right\rangle ;
\]
\item $\varphi$ is idempotent;
\item $\varphi\left(s,1\right)$ equals $\left(1,s\right)$ for every $s$
in $S_{+}$;
\item for every $\left(t,s,r\right)$ in $S_{+}^{3}$, applying any $\varphi_{i}$
to the triple $\varphi_{2}\varphi_{1}\varphi_{2}\left(t,s,r\right)$
leaves it unchanged, or $\varphi_{2}\varphi_{1}\varphi_{2}\left(t,s,r\right)$
contains $1$;
\item $\NF\left(t,s,r\right)=\NF\left(\varphi_{1}\left(t,s,r\right)\right)$
for all $\left(t,s,r\right)$ in $S_{+}^{3}$.
\end{enumerate}
Here, the map $\NF$ from $S^{\ast}$ to itself is defined inductively
on the length of the word, as follows:
\begin{itemize}
\item $\NF$ of the empty word is the empty word;
\item $\NF\left(s\right)$ equals $s$ for all $s$ in $S_{+}$;
\item $\NF$ of a word containing $1$ equals $\NF$ of the same word with
$1$ removed;
\item and
\[
\NF\left(s_{n},\ldots,s_{1}\right)\coloneqq\begin{cases}
\varphi_{n-1}\cdots\varphi_{1}\left(\NF\left(s_{n},\ldots,s_{2}\right),s_{1}\right), & \textrm{ if it contains no }1;\\
\NF\left(\varphi_{n-1}\cdots\varphi_{1}\left(\NF\left(s_{n},\ldots,s_{2}\right),s_{1}\right)\right), & \textrm{ otherwise}.
\end{cases}
\]
\end{itemize}
\begin{rem}
\label{rem:typo-phi(s)}Notice that there is a flaw in the above definition.
Since $\varphi$ is defined on the domain $S^{2}$, it is not clear
what should one do when one gets the expression of the form $\varphi\left(s\right)$,
where $s$ is a single element of $S_{+}$. However, such an expression
may indeed occur if $\varphi$ is applied just after $\NF$ has eliminated
$1$ and thus shortened the word in question. For instance, take $s_{2}=1$.
Then
\[
\NF\left(s_{2},s_{1}\right)=\varphi_{1}\left(\NF\left(s_{2}\right),s_{1}\right)=\varphi_{1}\left(s_{1}\right),
\]
which is not defined. We have resolved this issue by introducing Notation\uline{
}\ref{not:avoid-phi(s)}.
\end{rem}

\begin{rem}
\label{rem:terminology}Note that, in \cite{HO} and \cite{Ozo},
the map $\NF:S^{\ast}\to S^{\ast}$ is called the normal form. In
the present text (following \cite{DG}), however, a normal form is
a map from the presented monoid, not from the free monoid over a generating
set (recall Subsection~\ref{subsec:rewriting}). Accordingly, we
do not call the map $N_{\varphi},$ which is an analogue of $\NF$,
a normal form. Instead, we call such a syntactic transformation (of
arbitrary words into normal ones) a normalisation map (as in Definition~\ref{def:normalisation-map}).
\end{rem}

\subsection{\protect\label{subsec:factorability-rewriting}Rewriting system associated
with factorability}

A rewriting system is associated with a factorable monoid in a canonical
way.
\begin{lem}[{\cite[Lemma 5.1]{HO}}]
\label{lem:presentation-factorable-monoid}Let $\left(M,S,\eta\right)$
be a factorable monoid. If $R$ is the set of rewriting rules of the
form 
\[
\left(s,t\right)\rightarrow\eta\mu\left(s,t\right)
\]
for all $S_{+}$-words $s$ and $t$ such that $\left(s,t\right)$
is not stable, then $\left(S_{+},R\right)$ is a confluent, strongly
reduced rewriting system presenting $M$. Here, if $\overline{\eta}\left(st\right)=1$,
then the rewriting rule is interpreted as $\left(s,t\right)\rightarrow st$.
\end{lem}

\begin{rem}
\label{rem:different-presentations}Observe that, strictly speaking,
the presentation given by the property~(\ref{enu:local-factorability-structure-1})
of Definition~\ref{def:local-factorability-structure} (of local
factorability structure) is not the same as the one obtained by turning
rewriting rules into equations in Lemma~\ref{lem:presentation-factorable-monoid}.
Namely, if $st$ lies in $S$ (resp. equals $1$), then the former
contains the relation $\left(s,t\right)=\left(st,1\right)$ (resp.
$\left(s,t\right)=\left(1,1\right)$), whereas the latter has $\left(s,t\right)=st$
(resp. $\left(s,t\right)=1$). However, one can obtain the latter
from the former simply by removing the rightmost occurrence of $1$,
and\emph{ vice versa}.
\end{rem}

The associated rewriting system in Lemma~\ref{lem:presentation-factorable-monoid}
is not necessarily terminating, even if $S$ is finite and $M$ is
left-cancellative. The reader is referred to \cite[Appendix]{HO}
for an example of a factorable monoid whose associated rewriting system
is not terminating. The following result gives a sufficient condition
for the rewriting system associated with a factorable monoid to be
terminating. 
\begin{thm}[{\cite[Theorem 7.3]{HO}}]
\label{thm:stronger-conditions-terminating}Let $\left(M,S,\eta\right)$
be a factorable monoid. If the equalities
\[
\eta'\left(sf\right)=\eta'\left(s\cdot\eta'\left(f\right)\right),\quad\overline{\eta}\left(sf\right)=\overline{\eta}\left(s\cdot\eta'\left(f\right)\right)\cdot\overline{\eta}\left(f\right)
\]
hold for all $s$ in $S_{+}$ and $f$ in $M$, then the associated
rewriting system is terminating.
\end{thm}

\section{\protect\label{sec:quadratic}Quadratic normalisations}

This section recalls the notion of quadratic normalisation. After
presenting basic notions concerning normalisation in Subsection~\ref{subsec:normalisations-and-normal-forms},
we recollect the notion of quadratic normalisation in Subsection~\ref{subsec:quadratic}.
Then we focus on a particular class of quadratic normalisations in
Subsection~\ref{subsec:class-(4,3)}. The rewriting system associated
with a quadratic normalisation is recalled in Subsection~\ref{subsec:normalisation-rewriting}.
Finally, Subsection~\ref{subsec:left-weighted} recalls the notion
of left-weighted normalisation. For technical elaboration, see \cite{DG}.

\subsection{\protect\label{subsec:normalisations-and-normal-forms}Normalisation
and normal form}

Having already hinted in the Introduction that a normalisation is
a syntactic transformation of an arbitrary word into a normal one,
here we recall a formal definition.
\begin{defn}
\label{def:normalisation}Let $A$ be a set, and let $N$ be a map
from $A^{\ast}$ to itself. The pair $\left(A,N\right)$ is called
a \textbf{normalisation} if
\begin{enumerate}
\item \label{enu:normalisation-1}$N$ is length-preserving,
\item \label{enu:normalisation-2}restriction of $N$ to $A$ is the identity
map,
\item \label{enu:normalisation-3}the equality
\[
N\left(u|v|w\right)=N\left(u|N\left(v\right)|w\right)
\]
holds for all $A$-words $u$, $v$, $w$.
\end{enumerate}
\end{defn}

The map $N$ is called a \textbf{normalisation map}. An $A$-word
$w$ such that $N\left(w\right)=w$ is called \textbf{$N$-normal}.

A \textbf{normalisation for a monoid} $M$ is a normalisation $\left(S_{+},N\right)$
such that $M$ admits the presentation
\[
\left\langle S_{+}\vert\left\{ w=N\left(w\right)\vert w\in S_{+}^{\ast}\right\} \right\rangle .
\]

\begin{rem}
Recall that the usage of the letter $S$, as well as the subscript
$+$ in $S_{+}$ (or the absence of it), is explained by Convention~\ref{con:single-letter-for-1}.
The letter $S$ is not used in Definition~\ref{def:normalisation}
(of normalisation), in order to point out that normalisation itself
is a purely syntactic notion, as opposed to the notion of normalisation
for a monoid.
\end{rem}

\begin{example}[{\cite[Example 2.2]{DG}}]
\label{exa:lexicographic}Assume that $\left(A,<\right)$ is a totally
ordered nonempty finite set. Let $<^{\ast}$ denote the lexicographic
extension of $<$ to $A^{\ast}$. The image under $N$ of an $A$-word
$w$ is defined as the $<^{\ast}$-minimal word obtained by permuting
letters of $w$. Then $\left(A,N\right)$ is a normalisation for
the free commutative monoid over $A$.

\end{example}

Note that the definition of normalisation for a monoid $M$ implies
that there is a nontrivial monoid homomorphism, also known as grading,
from $M$ to the additive monoid of nonnegative integers, such that
the degree, i.e.\ image under grading, of every $s$ in $S_{+}$
equals $1$. Namely, set the degree of $f$ in $M$ to be the common
length of all the $S_{+}$-words representing $f$ (which is well-defined
due to the property~(\ref{enu:normalisation-1}) of normalisation).
Such monoids $\left(M,S_{+}\right)$ are called \textbf{graded}. In
other words, $M$ is graded with respect to a generating set $S_{+}$
if all $S_{+}$-words representing the same element of $M$ are equal
in length.

\begin{rem}[{\cite[Proposition 2.6]{DG}}]
\label{rem:normalisation-normal-form-graded}In a graded monoid,
a normalisation and a normal form are just different aspects of looking
at the same notion. Indeed, given a normalisation $\left(S_{+},N\right)$
for a monoid $M$, a normal form $\nf$ for $M$ with respect to $S_{+}$
is obtained by setting $\nf\left(f\right)=N\left(w\right)$ for any
$S_{+}$-word $w$ representing $f$. Conversely, if $\nf$ is a normal
form for a graded monoid $M$ with respect to a generating set $S_{+}$,
then one obtains a normalisation map by setting $N\left(w\right)=\nf\left(\ev\left(w\right)\right)$.
These two correspondences are inverse to each other.
\end{rem}

On the other hand, if a monoid $\left(M,S_{+}\right)$ is not graded,
i.e.\ if an element of $M$ may be represented by $S_{+}$-words
of different lengths, then a new letter representing $1$ can be introduced
to formally preserve length. For a normalisation $\left(S,N\right)$,
we say that an element $e$ of $S$ is \textbf{$N$-neutral} if the
equalities
\begin{equation}
N\left(w|e\right)=N\left(e|w\right)=N\left(w\right)|e\label{eq:neutral}
\end{equation}
hold for every $S$-word $w$. We say that $\left(S,N\right)$ is
a \textbf{normalisation mod $e$ for a monoid $M$} if $e$ is an
$N$-neutral element of $S$ and $M$ admits the presentation
\begin{equation}
\left\langle S\vert\left\{ w=N\left(w\right)\vert w\in S^{\ast}\right\} \cup\left\{ e=1\right\} \right\rangle .\label{eq:presentation_mod}
\end{equation}
Note that there can be at most one $N$-neutral element. We write
$\pi_{e}$ for the canonical projection from $S^{\ast}$ onto $S_{+}^{\ast}$,
which removes all the occurrences of $e$. This extends the equivalence
between normalisation and normal form from graded monoids to monoids
in general.
\begin{prop}[{\cite[Proposition 2.9]{DG}}]
\label{prop:normalisation-normal-form}~
\begin{enumerate}
\item If $\left(S,N\right)$ is a normalisation mod $e$ for a monoid $M$,
then a geodesic normal form $\nf$ for $M$ with respect to $S$ is
obtained by $\nf\left(f\right)=\pi_{e}\left(N\left(w\right)\right)$
for any $S$-word $w$ representing $f$.
\item \label{enu:normalisation-normal-form-2}Conversely, let $\nf$ be
a geodesic normal form for a monoid $M$ with respect to a generating
set $S_{+}$. Write $\ev^{e}$ for the extension of the evaluation
map $\ev:S_{+}^{\ast}\to M$ to $S^{\ast}$ by putting $\ev^{e}\left(e\right)=1$.
Then a normalisation $\left(S,N\right)$ mod $e$ for $M$ is provided
by the map 
\[
N\left(w\right)=\nf\left(\ev^{e}\left(w\right)\right)|e^{m},
\]
with $m$ denoting the number of letters $e$ to be added in order
to formally preserve length, namely $m=\left|w\right|-\left|\nf\left(\ev^{e}\left(w\right)\right)\right|$.
\item These two correspondences are inverse to each other.
\end{enumerate}
\end{prop}

\begin{rem}[{\cite[Remark 2.10]{DG}}]
\label{rem:graded-non-graded}If a monoid $\left(M,S_{+}\right)$
is graded and $\nf$ is a normal form on $\left(M,S_{+}\right)$,
then there are two normalisations arising from $\nf$, provided by
Remark~\ref{rem:normalisation-normal-form-graded} and Proposition~\ref{prop:normalisation-normal-form}(\ref{enu:normalisation-normal-form-2}),
respectively. To make a clear distinction, let us temporarily write
$N_{+}$ for the normalisation map of the former. The normalisations
$\left(S_{+},N_{+}\right)$ and $\left(S,N\right)$ are closely related
as the map $\pi_{e}\circ N$ is identically equal to the map $N_{+}\circ\pi_{e}$.
Therefore, it suffices to formally consider normalisations mod $e$
for monoids.
\end{rem}

\subsection{\protect\label{subsec:quadratic}Quadratic normalisations}

Let us extend Notation~\ref{not:at-positions}.

\begin{notation*}For a finite sequence of positive integers $u=\left(i_{1},\ldots,i_{n}\right)$,
we denote the composite map $F_{i_{n}}\circ\cdots\circ F_{i_{1}}$
(note that $F_{i_{1}}$ is applied first) by $F_{u}$, often omitting
commas in $u$ if all its components are single-digit numbers.\end{notation*}

If $\left(A,N\right)$ is a normalisation, let $\overline{N}$ denote
the restriction of $N$ to the set of all length-two $A$-words.
\begin{defn}
\label{def:quadratic}A normalisation $\left(A,N\right)$ is \textbf{quadratic}
if the following two requirements are met.
\begin{enumerate}
\item \label{enu:quadratic-1}An $A$-word $w$ is $N$-normal (meaning
$N\left(w\right)=w$) if, and only if, every length-two factor of
$w$ is $N$-normal.
\item \label{enu:quadratic-2}For every $A$-word $w$, there exists a finite
sequence of positions $u$, depending on $w$, such that $N\left(w\right)=\overline{N}_{u}\left(w\right)$.
\end{enumerate}
\end{defn}

\begin{example}
The normalisation given in Example~\ref{exa:lexicographic} is quadratic.
Indeed, a word is $<^{\ast}$-minimal if, and only if, its every length-two
factor is $<^{\ast}$-minimal; and each word can be ordered by switching
pairs of adjacent letters that are not ordered as expected.
\end{example}

An advantage of a quadratic normalisation is that it is completely
determined by the restriction $\overline{N}$.
\begin{prop}[{\cite[Proposition 3.6]{DG}}]
\label{prop:quadratic-presentation}~
\begin{enumerate}
\item \label{enu:quadratic-presentation-1}If $\left(S_{+},N\right)$ is
a quadratic normalisation for a monoid $M$, then $\overline{N}$
is idempotent and $M$ admits the presentation
\[
\left\langle S_{+}\vert\left\{ s|t=\overline{N}\left(s|t\right)\vert s,t\in S_{+}\right\} \right\rangle .
\]
\item \label{enu:quadratic-presentation-2}If $\left(A,N\right)$ is a quadratic
normalisation, then an element $e$ of $A$ is $N$-neutral if, and
only if, the following equalities hold for every $s$ in $A$:
\[
\overline{N}\left(s|e\right)=\overline{N}\left(e|s\right)=\overline{N}\left(s\right)|e.
\]
\item \label{enu:quadratic-presentation-3}If $\left(S,N\right)$ is a quadratic
normalisation mod $e$ for a monoid $M$, then $\overline{N}$ is
idempotent and $M$ admits the presentation
\[
\left\langle S_{+}\vert\left\{ s|t=\pi_{e}\left(\overline{N}\left(s|t\right)\right)\vert s,t\in S_{+}\right\} \right\rangle .
\]
\end{enumerate}
\end{prop}

If $\left(A,N\right)$ is a quadratic normalisation, then the image
under $N$ of an $A$-word is computed by sequentially applying $\overline{N}$
to length-two factors at various positions. For length-three words,
there are only two such positions. Since $\overline{N}$ is idempotent,
it suffices to consider alternating sequences of positions. This motivates
the notion of class, which measures the complexity of normalising
length-three words. For a nonnegative integer $m$, we write $12\left[m\right]$
(resp. $21\left[m\right]$) for the alternating sequence $121\ldots$
(resp. $212\ldots$) of length $m$. A quadratic normalisation $\left(A,N\right)$
is said to be of \textbf{left-class} $m$ (resp. \textbf{right-class}
$n$) if the equality $N\left(w\right)=\overline{N}_{12\left[m\right]}\left(w\right)$
(resp. $N\left(w\right)=\overline{N}_{21\left[n\right]}\left(w\right)$)
holds for every length-three $A$-word $w$. A quadratic normalisation
$\left(A,N\right)$ is of \textbf{class} $\left(m,n\right)$ if it
is of left-class $m$ and of right-class $n$.
\begin{example}
\label{exa:lexicographic-(3,3)}Let us consider the quadratic normalisation
given in Example~\ref{exa:lexicographic}. If $A$ has only one element,
then there is only one $A$-word of length $3$ and it is $N$-normal.
So, $\left(A,N\right)$ is of class $\left(0,0\right)$ in this case.

If $A$ has at least two elements, then, for every length-three $A$-word
$w$, the words $\overline{N}_{121}\left(w\right)$ and $\overline{N}_{212}\left(w\right)$
are $N$-normal. Namely, assuming that the word $a|b|c$ is normal,
compute $\overline{N}_{121}$ and $\overline{N}_{212}$ of $c|b|a$
(which provides the worst-case scenario):
\[
c|b|a\stackrel{\overline{N}_{1}}{\mapsto}b|c|a\stackrel{\overline{N}_{2}}{\mapsto}b|a|c\stackrel{\overline{N}_{1}}{\mapsto}a|b|c,\quad c|b|a\stackrel{\overline{N}_{2}}{\mapsto}c|a|b\stackrel{\overline{N}_{1}}{\mapsto}a|c|b\stackrel{\overline{N}_{2}}{\mapsto}a|b|c.
\]
Therefore, $\left(A,N\right)$ is of class $\left(3,3\right)$ in
this case. A smaller class cannot be obtained in general, as witnessed
by $\overline{N}_{12}\left(b|b|a\right)=b|a|b$ and $\overline{N}_{21}\left(b|a|a\right)=a|b|a$,
with $a$ and $b$ being any two $A$-letters such that $a<b$.
\end{example}

A class of a quadratic normalisation $\left(A,N\right)$ can be characterised
by relations involving only the restriction $\overline{N}$.
\begin{prop}[{\cite[Proposition 3.14]{DG}}]
\label{prop:class-characterisation}A quadratic normalisation $\left(A,N\right)$
is
\begin{enumerate}
\item of left-class $m$ if, and only if, the following three maps coincide
on $A^{3}$:
\[
\overline{N}_{12\left[m\right]},\quad\overline{N}_{12\left[m+1\right]},\quad\overline{N}_{21\left[m+1\right]}.
\]
\item \label{enu:class-characterisation-2}of right-class $n$ if, and only
if, the following three maps coincide on $A^{3}$:
\[
\overline{N}_{21\left[n\right]},\quad\overline{N}_{21\left[n+1\right]},\quad\overline{N}_{12\left[n+1\right]}.
\]
\item of class $\left(m,m\right)$ if, and only if, the map $\overline{N}_{12\left[m\right]}$
coincides with the map $\overline{N}_{21\left[m\right]}$ on $A^{3}$.
\end{enumerate}
\end{prop}

The \textbf{minimal left-class} of $\left(A,N\right)$ is the smallest
natural number $m$ such that $\left(A,N\right)$ is of left-class
$m$ if such $m$ exists, and $\infty$ otherwise. The minimal right-class
$n$ of $\left(A,N\right)$ is defined analogously. Then the minimal
class of $\left(A,N\right)$ is the pair $\left(m,n\right)$.
\begin{lem}[{\cite[Lemma 3.13]{DG}}]
\label{lem:minimal-class}The minimal class of quadratic normalisation
is either of the form $\left(m,n\right)$ with $\left|m-n\right|\leq1$,
or $\left(\infty,\infty\right)$.
\end{lem}

For example, the minimal class of the quadratic normalisation considered
in Example~\ref{exa:lexicographic-(3,3)} is $\left(3,3\right)$
if the generating set has at least two elements.

\subsection{\protect\label{subsec:class-(4,3)}Quadratic normalisations of class
$\left(4,3\right)$}

The class $\left(4,3\right)$ has particularly nice computational
properties (see \cite[Section 4]{DG}), not shared by higher classes
(see \cite[Example 3.23]{DG}), thanks to the diagrammatic tool called
the domino rule. Let $A$ be a set, and let $F$ be a map from $A^{2}$
to itself. We say that the \textbf{domino rule} is valid for $F$
if, for all $r_{1}$, $r_{2}$, $r_{1}'$, $r_{2}'$, $s_{0}$, $s_{1}$,
$s_{2}$ in $A$ such that $F\left(s_{0}|r_{1}\right)=r_{1}'|s_{1}$
and $F\left(s_{1}|r_{2}\right)=r_{2}'|s_{2}$, the following implication
holds: if $r_{1}|r_{2}$ and $r_{1}'|s_{1}$ and $r_{2}'|s_{2}$
are fixed points of $F$, then so is $r_{1}'|r_{2}'$. The domino
rule is expressed by the commutative diagram\begin{equation}\label{eq:domino}
\begin{tikzcd}[row sep=large, column sep=large] \bullet \arrow[r, shorten <=-2pt,  shorten >=-2pt, "r'_{1}"] \arrow[r, phantom, ""{name=s1'}', near end] \arrow[d, "s_{0}"'] & \bullet \arrow[r, shorten <=-2pt,  shorten >=-2pt, "r'_{2}"] \arrow[r, phantom, ""{name=s2'_1}, near start] \arrow[r, phantom, ""{name=s2'_2}', near end] \arrow[d, shorten <=-2pt,  shorten >=-2pt, "s_{1}"] \arrow[d, phantom, ""{name=t1}', near start] & \bullet \arrow[d, shorten <=-2pt,  shorten >=-2pt, "s_{2}"] \arrow[d, phantom, ""{name=t2}', near start] \\ \bullet \arrow[r, shorten <=-2pt,  shorten >=-2pt, "r_{1}"] \arrow[r, phantom, ""{name=s1}', near end] & \bullet \arrow[r, shorten <=-2pt, shorten >=-2pt, "r_{2}"] \arrow[r, phantom, ""{name=s2}', near start] & \bullet \arrow[from=s1', to=t1, shorten <=-2pt,  shorten >=-2pt, bend right, dash] \arrow[from=s1', to=s2'_1, shorten <=-2pt,  shorten >=-2pt, bend left=70, dash, dashed] \arrow[from=s2'_2, to=t2, shorten <=-2pt,  shorten >=-2pt, bend right, dash]  \arrow[from=s1, to=s2, shorten <=-2pt,  shorten >=-2pt, bend right=70, dash] \end{tikzcd}
\end{equation}where arcs denote fixed points of $F$: the solid ones are assumptions,
and the dashed one is the expected conclusion.
\begin{prop}[{\cite[Lemma 4.2]{DG}}]
\label{prop:quadratic-iff-domino}A quadratic normalisation $\left(A,N\right)$
is of class $\left(4,3\right)$ if, and only if, the domino rule
is valid for $\overline{N}$.

\end{prop}

The domino rule allows one to devise a simple universal recipe for
computing the images under a normalisation map. We refer the reader
to \cite[Proposition 4.4]{DG} for details, while here we only recall
the key step.
\begin{lem}[{\cite[Lemma 4.5]{DG}}]
\label{lem:normalisation-letter-next-to-normal-word}If $\left(A,N\right)$
is a quadratic normalisation of class $\left(4,3\right)$, then, for
every $s$ in A and every $N$-normal $A$-word $r_{1}|\cdots|r_{n}$,
we have
\[
N\left(s|r_{1}|\cdots|r_{n}\right)=\overline{N}_{1|2|\cdots|n-1|n}\left(s|r_{1}|\cdots|r_{n}\right).
\]
\end{lem}

\begin{rem}
\label{rem:leftmost-letter}Note, in particular, that Lemma~\ref{lem:normalisation-letter-next-to-normal-word}
implies that the leftmost letter of $N\left(s|r_{1}|\cdots|r_{n}\right)$
does not depend on $r_{2},\ldots,r_{n}$, but only on $s$ and $r_{1}$.
We will use this observation in Subsection~\ref{subsec:factorability-(4,3)}.
\end{rem}

By Proposition~\ref{prop:class-characterisation}(\ref{enu:class-characterisation-2}),
if $\left(A,N\right)$ is a quadratic normalisation of class $\left(4,3\right)$,
then the maps $\overline{N}_{1212}$, $\overline{N}_{212}$ and $\overline{N}_{2121}$
coincide on $A^{3}$. One of the major results of \cite{DG} is the
converse: every idempotent map satisfying such condition arises from
a quadratic normalisation of class $\left(4,3\right)$, in the following
sense.
\begin{prop}[{\cite[Proposition 4.7]{DG}}]
Let $A$ be a set, and let $F$ be a map from $A^{2}$ to itself.
If $F$ is idempotent, and if the maps $F_{1212}$, $F_{212}$ and
$F_{2121}$ coincide on $A^{3}$, then there is a quadratic normalisation
$\left(A,N\right)$ of class $\left(4,3\right)$ such that the map
$\overline{N}$ is identically equal to the map $F$.
\end{prop}

Quadratic normalisations of class $\left(4,3\right)$ are thus fully
axiomatised.

\subsection{\protect\label{subsec:normalisation-rewriting}Rewriting system associated
with quadratic normalisation}

There is a simple correspondence between quadratic normalisations
and quadratic rewriting systems.
\begin{prop}[{\cite[Proposition 3.7]{DG}}]
\label{prop:normalisation-rewriting}~
\begin{enumerate}
\item \label{enu:normalisation-rewriting-1}If $\left(S_{+},N\right)$ is
a quadratic normalisation for a monoid $M$, then a quadratic, reduced,
normalising and confluent rewriting system $\left(S_{+},R\right)$
presenting $M$, is obtained by defining $R$ as the set of rewriting
rules of the form 
\[
s|t\rightarrow\overline{N}\left(s|t\right)
\]
for all $s$ and $t$ in $S_{+}$ such that $s|t$ is not $N$-normal.
\item Conversely, if $\left(S_{+},R\right)$ is a quadratic, reduced, normalising
and confluent rewriting system presenting a monoid $M$, then $\left(S_{+},N\right)$
is a quadratic normalisation for $M$, with 
\[
N\left(w\right)=\widehat{w}
\]
for all $w$ in $S_{+}^{\ast}$.
\item These two correspondences are inverse to each other.
\end{enumerate}
\end{prop}

Proposition~\ref{prop:normalisation-rewriting}(\ref{enu:normalisation-rewriting-1})
can be adapted to a case where there is an $N$-neutral element. In
fact, an $N$-neutral element does not affect termination of the associated
rewriting system.
\begin{prop}[{\cite[Proposition 3.9]{DG}}]
\label{prop:normalisation-rewriting-e}~
\end{prop}

\begin{enumerate}
\item \label{enu:normalisation-rewriting-e-1}If $\left(S,N\right)$ is
a quadratic normalisation mod $e$ for a monoid $M$, then a reduced,
normalising and confluent rewriting system $\left(S_{+},R\right)$
presenting $M$, is obtained by defining $R$ as the set of rewriting
rules  of the form 
\[
s|t\rightarrow\pi_{e}\left(\overline{N}\left(s|t\right)\right)
\]
for all $s$ and $t$ in $S_{+}$ such that $s|t$ is not $N$-normal.
\item If the rewriting system in Proposition~\ref{prop:normalisation-rewriting}(\ref{enu:normalisation-rewriting-1})
terminates, then so does the one in (\ref{enu:normalisation-rewriting-e-1}). 
\end{enumerate}
The rewriting system associated with a quadratic normalisation in
Proposition~\ref{prop:normalisation-rewriting} need not be terminating
(see \cite[Proposition 5.7]{DG}). However, it is terminating for
quadratic normalisations of class $\left(4,3\right)$.
\begin{prop}[{\cite[Proposition 5.4]{DG}}]
\label{prop:(4,3)-terminates}\label{prop:(4,3)-converges}If $\left(A,N\right)$
is a quadratic normalisation of class $\left(4,3\right)$, then the
associated rewriting system is convergent. More precisely, every
rewriting sequence starting from an element of length $n$ has length
at most $2^{n}-n-1$.
\end{prop}

\subsection{\protect\label{subsec:left-weighted}Left-weighted normalisation}

Before closing this section, let us recall (from \cite[Subsection 6.2]{DG})
a notion to be used in the next section. A normalisation $\left(A,N\right)$
for a monoid $M$ is called \textbf{left-weighted} if, for all $s,t,s',t'$
in $A$, the equality $s'|t'=\overline{N}\left(s|t\right)$ implies
the left divisibility $s\preceq s'$ in $M$.
\begin{prop}[{\cite[Proposition 6.10]{DG}}]
\label{prop:garside-iff-left-weighted}Let $M$ be a left-cancellative
monoid $M$ containing no nontrivial invertible element. If $\left(S,N\right)$
is a quadratic normalisation mod $1$ for $M$, then the following
are equivalent.
\begin{itemize}
\item Every $N$-normal word $r_{1}|\cdots|r_{n}$ has the following property
for all $i<n$,
\begin{equation}
\textrm{for all }t\in S\textrm{ and }h\in M,\quad t\preceq hr{}_{i}r_{i+1}\textrm{ implies }t\preceq hr_{i}.\label{eq:greedy}
\end{equation}
\item The normalisation $\left(S,N\right)$ is of class $\left(4,3\right)$
and is left-weighted.
\end{itemize}
\end{prop}

The normal form $r_{1}|\cdots|r_{n}$ having the property~(\ref{eq:greedy})
is called \textbf{$S$-greedy} (at the $i$th position). It can be
expressed diagrammatically as follows. Commutativity of the diagram
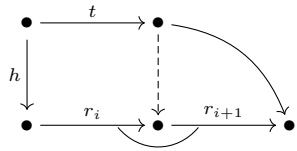
\[\begin{tikzcd}[row sep=large, column sep=large] \bullet \arrow[r, shorten <=-2pt,  shorten >=-2pt, "t"] \arrow[r, phantom, ""{name=s1'}', near end] \arrow[d, "h"'] & \bullet \arrow[rd, bend left, shorten <=-2pt,  shorten >=-2pt] \arrow[d, shorten <=-2pt, shorten >=-2pt, dashed] \arrow[d, phantom, ""{name=f1}', near start] & \\ \bullet \arrow[r, shorten <=-2pt,  shorten >=-2pt, "r_{i}"] \arrow[r, phantom, ""{name=s1}', near end] & \bullet \arrow[r, shorten <=-2pt,  shorten >=-2pt, "r_{i+1}"] \arrow[r, phantom, ""{name=s2}', near start] & \bullet \arrow[from=s1, to=s2, shorten <=-2pt,  shorten >=-2pt, bend right=50, dash] \end{tikzcd}\]without
the dashed arrow implies the existence of a dashed arrow making the
square (and thus also the triangle) commute. The arc joining $r_{i}$
and $r_{i+1}$ signifies that the word $r_{i}|r_{i+1}$ is greedy.
This notion is extensively studied in Garside theory. To an interested
reader, we recommend the monograph \cite{DDGKM} on the subject.

\section{\protect\label{sec:contribution}Correspondence between factorability
structures and quadratic normalisations}

In this section, a correspondence between factorability structures
and quadratic normalisations is established. Subsection~\ref{subsec:characterisation}
gives a characterisation of factorable monoids in terms of quadratic
normalisations. Subsection~\ref{subsec:factorability-(4,3)} shows
that, although a quadratic normalisation corresponding to a factorable
monoid is not of class $\left(4,3\right)$ in general, it is so if
a defining condition of local factorability structure is strengthened
in a suitable way. Finally, the strengthened definition is shown to
imply the additional assumption introduced in \cite{HO} (and recalled
in Theorem~\ref{thm:stronger-conditions-terminating}) in order to
reach termination of the associated rewriting system. This results
in a characterisation of quadratic normalisations of class $\left(4,3\right)$
in terms of factorability structures.

\subsection{\protect\label{subsec:characterisation}Characterisation of factorable
monoids}

In this subsection, a necessary and sufficient condition is given,
in terms of quadratic normalisations, for a monoid to be factorable.
This is achieved through a syntactic correspondence between a local
factorability structure and the restriction of a quadratic normalisation
map to length-two words. The property~(\ref{enu:local-factorability-structure-4})
of Definition~\ref{def:local-factorability-structure} (of local
factorability structure) is going to be essential in deriving our
main result so, for convenience, we are going to express it compactly
using the domino rule.
\begin{defn}
Let $\left(A,N\right)$ be a quadratic normalisation with the $N$-neutral
element $e$. The \textbf{weak domino rule} is valid for $\overline{N}$
if the domino rule is valid for $\overline{N}$ whenever none of the
elements $r_{1}'$, $r_{2}'$, $s_{2}$ of the diagram~(\ref{eq:domino})
equals $e$.
\end{defn}

Now, we can state the main result.
\begin{thm}
\label{thm:characterisation}A monoid $\left(M,S\right)$ admits
a factorability structure if, and only if, it admits a quadratic normalisation
$\left(N,S\right)$  mod $1$ such that the weak domino rule is valid
for $\overline{N}$. 
\end{thm}

The rest of this subsection presents a proof. First we verify that
a factorability structure yields a quadratic normalisation, rather
canonically. Assume that $\left(M,S,\eta\right)$ is a factorable
monoid. Since $N_{\varphi}$ is not length-preserving in general,
it does not make a suitable candidate for a normalisation map in the
sense of Subsection~\ref{subsec:normalisations-and-normal-forms}.
To repair this, let us introduce the following notation.

\begin{notation*}Let $\left(M,S,\eta\right)$ be a factorable monoid.
Denote by $N_{\varphi}'$ a pointwise length-preserving extended form
of $N_{\varphi}$, defined as follows:
\[
w\mapsto N_{\varphi}\left(w\right)|1^{m},\textrm{ with }m=\left|w\right|-\left|N_{\varphi}\left(w\right)\right|.
\]

\end{notation*}
\begin{lem}
\label{lem:factorable-implies-quadratic}If $\left(M,S,\eta\right)$
is a factorable monoid, then $\left(S,N_{\varphi}'\right)$ is a quadratic
normalisation mod $1$ for $M$.
\end{lem}

\begin{proof}
Recalling the relation between factorability structure $\eta$ and
local factorability structure $\varphi$, expressed by Theorem~\ref{thm:factorability-local-factorability},
let us check that $\left(S,N_{\varphi}'\right)$ has,\emph{ mutatis
mutandis}, the three properties of Definition~\ref{def:normalisation}
(of normalisation). Properties~(\ref{enu:normalisation-1}) and (\ref{enu:normalisation-2})
are satisfied by construction. Indeed, $N_{\varphi}'$ is length-preserving
and $N_{\varphi}'\left(s\right)=s$ for every $s$ in $S$, by definition.
Notice that Lemma~\ref{lem:normal-stable} implies the property~(\ref{enu:normalisation-3}).
Namely, since the normal form associated with the factorability structure
is everywhere stable, then so is $N_{\varphi}$, due to Theorem~\ref{thm:factorability-local-factorability}(\ref{enu:correspondence-eta-phi-4}).
This property is then transferred to $N_{\varphi}'$ as well, by construction.
In particular, the equality $N_{\varphi}'\left(u|v|w\right)=N_{\varphi}'\left(u|N_{\varphi}'\left(v\right)|w\right)$
holds for all $S$-words $u$, $v$, $w$.

Let us verify that the obtained normalisation $\left(S,N_{\varphi}'\right)$
is quadratic. The property~(\ref{enu:N_phi-3}) of Definition~\ref{def:normalisation-map}
(of the normalisation map $N_{\varphi}$) implies the property~(\ref{enu:quadratic-2})
of Definition~(\ref{def:quadratic}) (of quadratic normalisation)
and, consequently, also the right-to-left implication of the property~(\ref{enu:quadratic-1})
of Definition~\ref{def:quadratic}. The other direction of the property~(\ref{enu:quadratic-1})
of Definition~\ref{def:quadratic} follows from Lemma~\ref{lem:normal-stable}.

Finally, Lemma~\ref{lem:presentation-factorable-monoid}, together
with Remark~\ref{rem:different-presentations}, provides a presentation
showing that $\left(S,N_{\varphi}'\right)$ is a normalisation mod
$1$ for $M$.
\end{proof}
We say that $\left(S,N_{\varphi}'\right)$ is the \textbf{quadratic
normalisation corresponding to the factorability structure $\eta$}.

\begin{rem}
A note on formality is in order before we continue. As announced by
Remark\ \ref{rem:generating-(sub)set}, we have tried to respect
the two original conventions: of taking a generating set $S$ to be
a subset of a factorable monoid; and of distinguishing $S$ from its
image under an evaluation map arising from a normalisation. Now, however,
that we consider those quadratic normalisations $\left(S,N\right)$
that correspond to factorability structures, we interchangeably write
$e$ and $1$ for an $N$-neutral $S$-letter, depending on whether
we take the viewpoint of normalisation or factorability.
\end{rem}

\begin{rem}
\label{rem:equal-to-local-factorability}Note that, by construction,
the restriction $\overline{N_{\varphi}'}$ of $N_{\varphi}'$ to length-two
words is identically equal to the local factorability structure of
$\left(M,S,\eta\right)$. This fact will be often used implicitly.
\end{rem}

Having obtained a quadratic normalisation, we want to determine its
class. 
\begin{lem}
\label{lem:factorable-implies-right-class-4}If $\left(M,S,\eta\right)$
is a factorable monoid, then the quadratic normalisation corresponding
to $\eta$ is of class $\left(5,4\right)$.
\end{lem}

\begin{proof}
Denote by $\varphi$ the local factorability structure corresponding
to $\eta$ in the sense of Theorem~\ref{thm:factorability-local-factorability}.
If $\left(r,s,t\right)$ is an $S$-word, then Lemma~\ref{lem:2121-stable}
says that the word $N_{\varphi}'\left(r,s,t\right)$ equals the
word $\varphi_{2121}\left(r,s,t\right)$. We conclude that quadratic
normalisation $\left(S,N_{\varphi}'\right)$ is of right-class $4$.
Then Lemma~\ref{lem:minimal-class} grants the left-class $5$.
\end{proof}
The following example demonstrates that the minimal right-class of
a quadratic normalisation corresponding to a factorability structure
is not smaller than $4$, in general.
\begin{example}[{\cite[Example 2.1.13]{Hes}}]
\label{exa:sign}Consider the monoid $\left(\mathbb{Z},+\right)$
with respect to the generating set $\left\{ -1,+1\right\} $. The
factorisation map is defined by
\[
g\mapsto\left(\sgn\left(g\right),g-\sgn\left(g\right)\right),
\]
where $\sgn:\mathbb{Z}\to\left\{ -1,0,+1\right\} $ denotes the sign
function. One can check that this is a factorable monoid.\footnote{Technically speaking, this is a factorable group by \cite[Example 3.2.2]{Vis},
and a weakly factorable monoid. Hence, it is also a factorable monoid
by \cite[Proposition 2.1.28]{Hes}.} Note that $\varphi_{212}\left(1,-1,-1\right)$ equals $\left(0,-1,0\right)$,
whereas $\varphi_{2121}\left(1,-1,-1\right)$ equals $\left(-1,0,0\right)$.
Therefore, the minimal right-class of the corresponding quadratic
normalisation is at least $4$; then it is exactly $4$, by Lemma~\ref{lem:factorable-implies-right-class-4}.
\end{example}

The next example, adapted from \cite[Appendix, Proposition .7]{HO},
shows that the minimal left-class of a quadratic normalisation corresponding
to a factorability structure is not smaller than $5$ in general.
\begin{example}
\label{exa:right-class-at-least-4}Let $S_{+}$ be the following set:
\[
\left\{ a_{1},a_{2},b_{1},b_{2},b_{3},b_{4},b_{5},b_{6},c_{1},c_{2},c_{3},c_{4},c_{5},c_{6},d_{1},d_{2},e_{2},e_{3},f_{2},f_{3},g_{2},g_{3},h_{2},h_{3},i,j,k\right\} .
\]
Define a map $\varphi:S^{2}\to S^{2}$ as follows:
\begin{gather*}
\left(b_{1},a_{1}\right)\stackrel{\varphi}{\mapsto}\left(b_{2},a_{2}\right),\quad\left(c_{1},b_{2}\right)\stackrel{\varphi}{\mapsto}\left(c_{2},b_{3}\right),\quad\left(d_{1},c_{2}\right)\stackrel{\varphi}{\mapsto}\left(d_{2},c_{3}\right),\quad\left(c_{3},b_{3}\right)\stackrel{\varphi}{\mapsto}\left(c_{4},b_{4}\right),\\
\left(b_{4},a_{2}\right)\stackrel{\varphi}{\mapsto}\left(b_{5},a_{1}\right),\quad\left(c_{4},b_{5}\right)\stackrel{\varphi}{\mapsto}\left(c_{5},b_{6}\right),\quad\left(d_{2},c_{5}\right)\stackrel{\varphi}{\mapsto}\left(d_{1},c_{6}\right),\quad\left(c_{6},b_{6}\right)\stackrel{\varphi}{\mapsto}\left(c_{1},b_{1}\right),\\
\left(b_{3},a_{2}\right)\stackrel{\varphi}{\mapsto}\left(e_{2},1\right),\quad\left(b_{6},a_{1}\right)\stackrel{\varphi}{\mapsto}\left(e_{3},1\right),\quad\left(c_{2},e_{2}\right)\stackrel{\varphi}{\mapsto}\left(g_{2},f_{2}\right),\quad\left(c_{5},e_{3}\right)\stackrel{\varphi}{\mapsto}\left(g_{3},f_{3}\right),\\
\left(c_{3},e_{2}\right)\stackrel{\varphi}{\mapsto}\left(g_{3},f_{3}\right),\quad\left(c_{6},e_{3}\right)\stackrel{\varphi}{\mapsto}\left(g_{2},f_{2}\right)\quad\left(d_{1},g_{2}\right)\stackrel{\varphi}{\mapsto}\left(i,h_{2}\right),\quad\left(d_{2},g_{3}\right)\stackrel{\varphi}{\mapsto}\left(i,h_{3}\right),\\
\left(h_{2},f_{2}\right)\stackrel{\varphi}{\mapsto}\left(k,j\right),\quad\left(h_{3},f_{3}\right)\stackrel{\varphi}{\mapsto}\left(k,j\right)\quad\left(1,s\right)\stackrel{\varphi}{\mapsto}\left(s,1\right)\textrm{ for all }s\in S,
\end{gather*}
and $\left(s,t\right)\stackrel{\varphi}{\mapsto}\left(s,t\right)$
if $\left(s,t\right)$ is not listed above.

By \cite[Appendix, Proposition .7]{HO}, the map $\varphi$ is a local
factorability structure. Let us determine the class of the corresponding
quadratic normalisation.

Computing $\varphi_{1212}\left(c_{1},b_{1},a_{1}\right)$ gives 
\[
\left(c_{1},b_{1},a_{1}\right)\stackrel{\varphi_{1}}{\mapsto}\left(c_{1},b_{1},a_{1}\right)\stackrel{\varphi_{2}}{\mapsto}\left(c_{1},b_{2},a_{2}\right)\stackrel{\varphi_{1}}{\mapsto}\left(c_{2},b_{3},a_{2}\right)\stackrel{\varphi_{2}}{\mapsto}\left(c_{2},e_{2},1\right),
\]
whereas
\[
\varphi_{12121}\left(c_{1},b_{1},a_{1}\right)=\varphi_{1}\left(c_{2},e_{2},1\right)=\left(g_{2},f_{2},1\right).
\]
Hence the minimal left-class is at least $5$.

The above computation also gives $\varphi_{212}\left(c_{1},b_{1},a_{1}\right)=\left(c_{2},e_{2},1\right)$,
whereas
\[
\varphi_{2121}\left(c_{1},b_{1},a_{1}\right)=\varphi_{1}\left(c_{2},e_{2},1\right)=\left(g_{2},f_{2},1\right).
\]
Thus, the minimal right-class is at least $4$, as expected according
to Lemma~\ref{lem:minimal-class}.
\end{example}

The previous two examples witness that the estimate of class in Lemma~\ref{lem:factorable-implies-right-class-4}
is as good as one can hope for. On the other hand, observe that not
every quadratic normalisation of class $\left(5,4\right)$ is corresponding
to a factorability structure, as demonstrated by the following example,
adapted from \cite[Example 3.15]{DG}.
\begin{example}
\label{exa:suggests-candidate}Let $A=\left\{ a,b_{1},b_{2},b_{3},b_{4},b_{5}\right\} $,
and let $R$ consist of the rules $ab_{i}\rightarrow ab_{i+1}$ for
$i<5$ even, and $b_{i}a\rightarrow b_{i+1}a$ for $i<5$ odd. The
rewriting system $\left(A,R\right)$ is clearly quadratic and reduced.
Notice that it is also terminating because each rewriting rule only
increases the index of a letter $b$ in a word. Furthermore, it is
confluent, as illustrated by the following diagram (it suffices to
investigate the so-called critical pairs, see e.g.\ \cite[Section 1]{LP}):
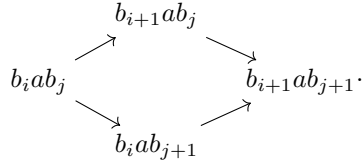
\[
\begin{tikzcd}[cramped, row sep=small, column sep=small] & b_{i+1}ab_{j} \arrow[rd, bend left=0] \\ b_{i}ab_{j} \arrow[ru, bend left=0] \arrow[rd, bend right=0] & & b_{i+1}ab_{j+1} \\ & b_{i}ab_{j+1} \arrow[ru, bend right=0] \end{tikzcd}.
\]

Denote by $\left(A,N\right)$ the quadratic normalisation associated
with $\left(A,R\right)$ by Proposition~\ref{prop:normalisation-rewriting}.
Let us determine the minimal class of $\left(A,N\right)$. If a length-three
word does not begin and end with the letter $a$, then it is either
$N$-normal or it becomes $N$-normal in a single step. On the other
hand, normalising $a|b_{1}|a$ takes four steps starting from the
right (or five steps starting from the left):

\[
a|b_{1}|a\left(\stackrel{\overline{N}_{1}}{\mapsto}a|b_{1}|a\right)\stackrel{\overline{N}_{2}}{\mapsto}a|b_{2}|a\stackrel{\overline{N}_{1}}{\mapsto}a|b_{3}|a\stackrel{N_{2}}{\mapsto}a|b_{4}|a\stackrel{N_{1}}{\mapsto}a|b_{5}|a.
\]
The minimal class of $\left(A,N\right)$ is thus $\left(5,4\right)$.

However, the normalisation $\left(A,N\right)$ is not corresponding
to any factorability structure. To see this, observe that 
\[
\overline{N}_{212}\left(a,b_{1},a\right)=\left(a,b_{4},a\right),
\]
whereas 
\[
\overline{N}_{2121}\left(a,b_{1},a\right)=\left(a,b_{5},a\right).
\]
Therefore, $\overline{N}$ fails to admit the property~(\ref{enu:local-factorability-structure-4})
of Definition~\ref{def:local-factorability-structure} (of local
factorability structure). We conclude (by Remark~\ref{rem:equal-to-local-factorability})
that $\overline{N}$ does not correspond to a factorability structure.
\end{example}

A natural question to ask is: among quadratic normalisations of class
$\left(5,4\right)$, what distinguishes those that correspond to a
factorability structure? Example~\ref{exa:suggests-candidate} suggests
a candidate (by what it is lacking): simply impose the weak domino
rule upon quadratic normalisation of class $\left(5,4\right)$. Before
testing sufficiency of these two properties combined, let us notice
that they are not independent from each other.
\begin{lem}
\label{lem:weak-domino-implies-(5,4)}Let $\left(S,N\right)$ be a
quadratic normalisation having an $N$-neutral element $e$. If the
weak domino rule is valid for $\overline{N}$, then $\left(S,N\right)$
is of class $\left(5,4\right)$.
\end{lem}

\begin{proof}
Let $\left(r,s,t\right)$ be an $S$-word. If 
\[
\overline{N}_{2121}\left(r,s,t\right)=\overline{N}_{212}\left(r,s,t\right),
\]
then
\[
N\left(r,s,t\right)=\overline{N}_{212}\left(r,s,t\right),
\]
so it takes at most three steps starting with $\overline{N}_{2}$
to normalise $\left(r,s,t\right)$.

Otherwise, $\overline{N}_{212}\left(r,s,t\right)$ contains $e$ by
the weak domino rule. Denote $\left(a,b,c\right)\coloneqq\overline{N}_{212}\left(r,s,t\right).$
\begin{casenv}
\item If $e$ occurs exactly once in the triple $\left(a,b,c\right)$, then
it cannot be at the leftmost position in $\overline{N}_{21}\left(r,s,t\right)$.
So, it has to be at the rightmost position in $\overline{N}_{212}\left(r,s,t\right)$.
In other words, $c$ has to be equal to $e$. Consequently,
\begin{equation}
N\left(a,b,c\right)=\overline{N}_{1}\left(a,b,c\right)=\overline{N}_{2121}\left(r,s,t\right).\label{eq:weak-domino-implies-(5,4)}
\end{equation}
\item If $e$ occurs exactly twice in the triple $\left(a,b,c\right)$,
then either $a\neq e$ or $b\neq e$. In each case, the equalities~(\ref{eq:weak-domino-implies-(5,4)})
hold.
\item If $e$ occurs three times in the triple $\left(a,b,c\right)$, then
clearly the equalities~(\ref{eq:weak-domino-implies-(5,4)}) hold.
\end{casenv}
Therefore, $\left(S,N\right)$ is of right class $4$, hence of class
$\left(5,4\right)$, by Lemma~\ref{lem:minimal-class}.
\end{proof}
The following proposition shows that adding the weak domino rule
to the properties of quadratic normalisation (and thus granting the
class $\left(5,4\right)$ by Lemma~\ref{lem:weak-domino-implies-(5,4)})
suffices to yield factorability.
\begin{prop}
\label{prop:normalisation-yields-local-factorability}Let $\left(S,N\right)$
be a quadratic normalisation mod $e$ for a monoid $M$. If the weak
domino rule is valid for $\overline{N}$, then \textup{$\overline{N}$}
is a local factorability structure.
\end{prop}

\begin{proof}
The following list shows that the map $\overline{N}$ has, \emph{mutatis
mutandis}, all the properties of Definition~\ref{def:local-factorability-structure}
(of local factorability structure).
\begin{enumerate}
\item By Proposition~\ref{prop:quadratic-presentation}(\ref{enu:quadratic-presentation-3}),
$M$ admits the presentation
\[
\left\langle S_{+}\vert\left\{ s|t=\pi_{e}\left(\overline{N}\left(s|t\right)\right)\vert s,t\in S_{+}\right\} \right\rangle .
\]
Although this presentation is not the same as the one in Definition~\ref{def:local-factorability-structure},
the two are equivalent (see Remark~\ref{rem:different-presentations}).
\item By the property~(\ref{enu:normalisation-3}) of Definition~\ref{def:normalisation}
(of normalisation), $\overline{N}$ is idempotent.
\item By Proposition~\ref{prop:quadratic-presentation}(\ref{enu:quadratic-presentation-2}),
the equality $\overline{N}\left(e|s\right)=s|e$ holds for all $s$
in $S_{+}$.
\item This property is explicitly assumed. 
\item By Lemma~\ref{lem:weak-domino-implies-(5,4)}, the normalisation
$\left(S,N\right)$ is of class $\left(5,4\right)$. By the definition
of right-class $4$, the $N$-normal word $N\left(r,s,t\right)$ equals
$\overline{N}_{2121}\left(r,s,t\right)$, which further equals $\overline{N}_{12121}\left(r,s,t\right)$
by the definition of left-class $5$. Hence, $N\left(r,s,t\right)$
equals $\left(N\circ\overline{N}_{1}\right)\left(r,s,t\right)$ for
all $\left(r,s,t\right)$ in $S_{+}^{3}$.
\end{enumerate}
\end{proof}
We have, thus, proved one direction of Theorem~\ref{thm:characterisation}.
The other one follows from Remark~\ref{rem:equal-to-local-factorability},
as the restriction to length-two words of a quadratic normalisation
corresponding to a factorability structure has all the properties
of the corresponding local factorability structure and, in particular,
the weak domino rule is valid. Thereby, we have completed the proof
of Theorem~\ref{thm:characterisation}.

The following corollary is an immediate consequence.
\begin{cor}
\label{cor:immediate-consequence}~
\begin{enumerate}
\item Associating a factorability structure to a quadratic normalisation
such that the weak domino rule is valid, and associating a quadratic
normalisation to a factorability structure (the weak domino rule being
valid automatically), as given above, are inverse transformations.
\item \label{enu:immediate-consequence-2}The normal forms associated with
a factorability structure and the corresponding quadratic normalisation
are the same.
\item \label{enu:immediate-consequence-3}The rewriting systems associated
with a factorability structure and the corresponding quadratic normalisation
are equivalent, the only difference being dummy letters to preserve
length in the latter.
\end{enumerate}
\end{cor}

\subsection{\protect\label{subsec:factorability-(4,3)}Factorability in relation
to quadratic normalisation of class $\left(4,3\right)$}

\uline{}Having established a general correspondence between a factorability
structure and a quadratic normalisation in the previous subsection,
we are now going to further elaborate these links in the case when
the quadratic normalisation involved is of class $\left(4,3\right)$.
First we emphasise a particular, yet important, consequence of Proposition~\ref{prop:normalisation-yields-local-factorability}.
\begin{cor}
\label{cor:(4,3)-implies-factorable}If $\left(S,N\right)$ is a quadratic
normalisation of class $\left(4,3\right)$ mod $1$ for a monoid $M$,
then \textup{$\overline{N}$} is a local factorability structure.
\end{cor}

Although the converse does not hold in general, it does so in the
case of graded monoids (as defined in Subsection~\ref{subsec:normalisations-and-normal-forms}).
\begin{lem}
Let $\left(M,S,\eta\right)$ be a factorable monoid. If $\left(M,S_{+}\right)$
is graded, then $\left(S_{+},N_{\varphi}\right)$ is a quadratic normalisation
of class $\left(4,3\right)$.
\end{lem}

\begin{proof}
In addition to the conclusion of Lemma~\ref{lem:factorable-implies-quadratic}
that $\left(S_{+},N_{\varphi}\right)$ is a quadratic normalisation,
let us show how the existence of grading ensures the right-class $3$.
By the property~(\ref{enu:local-factorability-structure-4}) of Definition~\ref{def:local-factorability-structure}
(of local factorability structure), for every $\left(r,s,t\right)$
in $S_{+}^{3}$, the equality
\[
\left(N_{\varphi}\right)_{2121}\left(r,s,t\right)=\left(N_{\varphi}\right)_{212}\left(r,s,t\right)
\]
holds or $\left(N_{\varphi}\right)_{212}\left(r,s,t\right)$ contains
$1$ but, since $M$ is graded, the latter case cannot occur. Therefore,
$\left(S_{+},N_{\varphi}\right)$ is of right-class $3$. Then the
property of being of left-class $4$ follows from Lemma~\ref{lem:minimal-class}.
\end{proof}
Let us pin down those defining properties of quadratic normalisation
of class $\left(4,3\right)$ mod $1$ for $M$ that do not necessarily
arise from a factorability structure on $M$. In other words, we are
looking for a (not too restrictive) property that would complement
a factorability structure to a quadratic normalisation of class $\left(4,3\right)$.
The property~(\ref{enu:local-factorability-structure-4}) of Definition~\ref{def:local-factorability-structure}
(of local factorability structure) grants the equality
\begin{equation}
\varphi_{2121}\left(r,s,t\right)=\varphi_{212}\left(r,s,t\right)\label{eq:right-class-3}
\end{equation}
only for $S_{+}$-words $\left(r,s,t\right)$ such that $\varphi_{212}\left(r,s,t\right)$
contains no $1$. On the other hand, the right-class $3$ (i.e.\
the domino rule) requires the equality~(\ref{eq:right-class-3})
to hold for every $S$-word $\left(r,s,t\right)$, regardless of
whether $\varphi_{212}\left(r,s,t\right)$ contains $1$ or not.
Therefore, in order for a factorability structure to induce a quadratic
normalisation of class $\left(4,3\right)$, it suffices to strengthen
the condition~(\ref{enu:local-factorability-structure-4}) of Definition~\ref{def:local-factorability-structure},
as follows.
\begin{prop}
\label{prop:stronger-factorable-implies-(4,3)}Let $\left(M,S,\eta\right)$
be a factorable monoid. If the equality~(\ref{eq:right-class-3})
holds for each $S_{+}$-word $\left(r,s,t\right)$ such that $\varphi_{212}\left(r,s,t\right)$
contains $1$, then $\left(S,N_{\varphi}'\right)$ is a quadratic
normalisation of class $\left(4,3\right)$ mod $1$ for $M$.
\end{prop}

\begin{proof}
Lemma~\ref{lem:factorable-implies-quadratic} says that $\left(S,N_{\varphi}'\right)$
is a quadratic normalisation mod $1$ for $M$.

To obtain the class $\left(4,3\right)$, what remains to be shown
is that the equality~(\ref{eq:right-class-3}) holds for all $\left(r,s,t\right)$
in $\left(S\right)^{3}\setminus S_{+}^{3}$. If $t$ equals $1$,
then $\varphi_{1}\left(r,s,t\right)$ equals $\varphi\left(r,s\right)|1$,
which is everywhere stable. If $s$ equals $1$, then $\varphi_{21}\left(r,s,t\right)$
equals $\varphi\left(r,t\right)|1$, which is everywhere stable.
If $r$ equals $1$, then $\varphi_{121}\left(r,s,t\right)$ equals
$\varphi_{212}\left(r,s,t\right)$ which equals $\varphi\left(s,t\right)|1$,
which is, again, everywhere stable.
\end{proof}
We have shown that a quadratic normalisation of class $\left(4,3\right)$
yields a factorability structure (Corollary~\ref{cor:(4,3)-implies-factorable}),
but not \emph{vice versa} (Examples~\ref{exa:sign} and \ref{exa:right-class-at-least-4}
or Theorem~\ref{thm:characterisation}). However, a factorability
structure does yield a quadratic normalisation of class $\left(4,3\right)$
under a stronger condition on local factorability (Proposition~\ref{prop:stronger-factorable-implies-(4,3)}).
Therefore, under the same condition, the rewriting system associated
with a factorable monoid is terminating, by Proposition~\ref{prop:(4,3)-converges}
and Corollary~\ref{cor:immediate-consequence}(\ref{enu:immediate-consequence-3}).

From another point of view, recall that Theorem~\ref{thm:stronger-conditions-terminating}
ensures termination of the rewriting system associated with a factorable
monoid, under an additional assumption on the factorisation map. It
is then natural to ask what is the relation between the additional
condition of Proposition~\ref{prop:stronger-factorable-implies-(4,3)}
and the additional assumption of Theorem~\ref{thm:stronger-conditions-terminating},
which are both known to ensure termination of the associated rewriting
system.

In the rest of the present subsection, we investigate the relation
between these two optional properties of a factorable monoid $\left(M,S,\eta\right)$:
for each $S_{+}$-word $\left(r,s,t\right)$ such that $\varphi_{212}\left(r,s,t\right)$
contains $1$, Condition
\begin{equation}
\left(\eta\mu\right)_{2121}\left(r,s,t\right)=\left(\eta\mu\right)_{212}\left(r,s,t\right);\label{eq:stable_212_on_S^3}
\end{equation}
and, for all $s$ in $S_{+}$ and $f$ in $M$, Assumption
\begin{equation}
\left(sf\right)'=\left(sf'\right)',\quad\overline{sf}=\overline{sf'}\cdot\overline{f}.\label{eq:the_stronger_conditions}
\end{equation}

\begin{rem}
Note that Assumption~(\ref{eq:the_stronger_conditions}) is trivially
valid in the case where $f$ lies in $S$ or in the case where $\left(s,f\right)$
is stable.
\end{rem}

It has already been known that Assumption~(\ref{eq:the_stronger_conditions})
implies Condition~(\ref{eq:stable_212_on_S^3}), as follows.
\begin{lem}[{\cite[Lemma 7.1]{HO}}]
\label{lem:stronger-conditions-imply-stable-212-on-M^3}Let $\left(M,S,\eta\right)$
be a factorable monoid. If Assumption~(\ref{eq:the_stronger_conditions})
is valid for all $s$ in $S_{+}$ and $f$ in $M$, then the maps
$\left(\eta\mu\right)_{212}$ and $\left(\eta\mu\right)_{2121}$ coincide
on $M^{3}$. 
\end{lem}

\begin{cor}
\label{cor:stronger-conditions-imply-stable-212-on-S^3}Let $\left(M,S,\eta\right)$
be a factorable monoid. If Assumption~(\ref{eq:the_stronger_conditions})
is valid for all $s$ in $S_{+}$ and $f$ in $M$, then Condition~(\ref{eq:stable_212_on_S^3})
is satisfied for all $\left(r,s,t\right)$ in $S_{+}^{3}$.
\end{cor}

For the reader's convenience, we adapt the proof of Lemma~\ref{lem:stronger-conditions-imply-stable-212-on-M^3}
here, to prove Corollary~\ref{cor:stronger-conditions-imply-stable-212-on-S^3}.
\begin{proof}
If Assumption~(\ref{eq:the_stronger_conditions}) is valid for all
$s$ in $S_{+}$ and $f$ in $M$, then the maps $\eta_{1}\mu_{1}$
and $\mu_{2}\eta_{1}\mu_{1}\eta_{2}$ coincide on $S_{+}\times M$.
Composing each of these two maps with $\mu_{2}$ and then composing
$\eta_{2}$ with the obtained composite map, we see that the maps
$\eta_{2}\eta_{1}\mu_{1}\mu_{2}$ and $\eta_{2}\mu_{2}\eta_{1}\mu_{1}\eta_{2}\mu_{2}$
coincide on $S_{+}\times M^{2}$. Note that the map $\eta_{2}\eta_{1}$
produces the $\eta$-normal form by definition, and that the restriction
of the map $\eta_{2}\mu_{2}\eta_{1}\mu_{1}\eta_{2}\mu_{2}$ to $S_{+}^{3}$
coincides with the map $\varphi_{2}\varphi_{1}\varphi_{2}$. Therefore,
every image under $\varphi_{2}\varphi_{1}\varphi_{2}$ is everywhere
stable by Lemma~\ref{lem:normal-stable}. Thus, Condition~(\ref{eq:stable_212_on_S^3})
is satisfied for all $\left(r,s,t\right)$ in $S_{+}^{3}$.
\end{proof}
In the opposite direction, a partial result has already been known.
Namely, Condition~(\ref{eq:stable_212_on_S^3}) implies Assumption~(\ref{eq:the_stronger_conditions})
under certain, quite restrictive, additional requirements imposed
on both the monoid and the normalisation. The next lemma is a
straightforward adaptation of \cite[Proposition IV.1.49]{DDGKM};
it is also hinted in the proof of \cite[Corollary 7.4.5]{Ozo}.
\begin{rem}
Relying on Theorem~\ref{thm:characterisation} and Corollary~\ref{cor:immediate-consequence}(\ref{enu:immediate-consequence-2})
in particular, we are going to abuse terminology by saying \textquoteleft the
normal form\textquoteright{} without specifying whether it arises
from factorability or normalisation.
\end{rem}

\begin{lem}
Let $M$ be a left-cancellative monoid containing no nontrivial invertible
element. If $\left(S,N\right)$ is a left-weighted quadratic normalisation
of class $\left(4,3\right)$ mod $1$ for $M$, then Assumption~(\ref{eq:the_stronger_conditions})
is valid for all $s$ in $S_{+}$ and $f$ in $M$.
\end{lem}

\begin{proof}
We need to show that the equalities $\left(sf\right)'=\left(sf'\right)'$
and $\overline{sf}=\overline{sf'}\cdot\overline{f}$ hold for all
$s$ in $S_{+}$ and $f$ in $M$.

By the definition of factorability structure, we have $\left(sf'\right)'\preceq sf'$,
and $f'\preceq f$ hence also $sf'\preceq sf$. The transitivity gives
$\left(sf'\right)'\preceq sf$, and the assumption that $\left(S,N\right)$
is left-weighted then yields $\left(sf'\right)'\preceq\left(sf\right)'$.

By Corollary~\ref{cor:immediate-consequence}(\ref{enu:immediate-consequence-2}),
the normal form of $f$ has the form $f'|r_{2}|\cdots|r_{n}$. Hence,
$\left(sf\right)'\preceq sf=sf'r_{2}\cdots r_{n}$. Since $\left(sf\right)'$
lies in $S$ and $f'|r_{2}|\cdots|r_{n}$ is normal, we obtain $\left(sf\right)'\preceq sf'$
by the property~(\ref{eq:greedy}) in Proposition~\ref{prop:garside-iff-left-weighted}.
That yields $\left(sf\right)'\preceq\left(sf'\right)'$ due to the
assumption that $\left(S,N\right)$ is left-weighted. 

Since $M$ is a left-cancellative monoid containing no nontrivial
invertible element, we conclude that $\left(sf\right)'=\left(sf'\right)'$.
Then $\overline{sf}=\overline{sf'}\cdot\overline{f}$ follows from
the left cancellation property.
\end{proof}
We want to find out whether Condition~(\ref{eq:stable_212_on_S^3})
implies Assumption~(\ref{eq:the_stronger_conditions}) in general,
i.e.\ without all the additional requirements of the previous lemma,
or is the latter strictly stronger than the former. Let us begin by
considering length-two words.
\begin{lem}
\label{lem:stable-212-on-S^3-implies-stronger-conditions-for-=00005Cell(g)=00003D2}Let
$\left(M,S,\eta\right)$ be a factorable monoid. If Condition~(\ref{eq:stable_212_on_S^3})
is satisfied for each $S_{+}$-word $\left(r,s,t\right)$ such that
$\varphi_{212}\left(r,s,t\right)$ contains $1$, then Assumption~(\ref{eq:the_stronger_conditions})
is valid for all $s$ in $S_{+}$ and length-two $f$ in $M$.
\end{lem}

\begin{proof}
The idea is to equate two different expressions of the normal form
of $sf$. Fix arbitrary $s$ in $S_{+}$ and a length-two element
$f$ of $M$.

First compute the $\eta$-normal form of $sf$, by definition:
\begin{equation}
sf\stackrel{\eta_{1}}{\mapsto}\left(\left(sf\right)',\overline{sf}\right)\stackrel{\eta_{2}}{\mapsto}\left(\left(sf\right)',\overline{sf}',\overline{\overline{sf}}\right).\label{eq:NF(sf)}
\end{equation}

Next we compute the same normal form in a different manner. Since
the length of $f$ equals $2$, we know that $\overline{f}$ lies
in $S_{+}$. Hence, $\left(s,f',\overline{f}\right)$ lies in $S_{+}^{3}$.
So we have another length-three $S_{+}$-word evaluating to $sf$.
By the property~(\ref{enu:local-factorability-structure-4}) of Definition~\ref{def:local-factorability-structure}
(of local factorability structure) and Condition~(\ref{eq:stable_212_on_S^3}),
the word $\left(\eta\mu\right)_{212}\left(s,f',\overline{f}\right)$
is everywhere stable, hence also normal by Lemma~\ref{lem:factorable-implies-quadratic}
and the property~(\ref{enu:quadratic-1}) of Definition~\ref{def:quadratic}
(of quadratic normalisation). Therefore, $\left(\eta\mu\right)_{212}\left(s,f',\overline{f}\right)$
is the normal form of the evaluation of $\left(s,f',\overline{f}\right)$,
i.e.\ of $sf$. Since computing $\left(\eta\mu\right)_{212}\left(s,f',\overline{f}\right)$
depends on whether any $1$ occurs in the process, we consider the
following cases.
\begin{casenv}
\item If $sf'$ is not an element of $S$, then
\[
\left(s,f',\overline{f}\right)\stackrel{\left(\eta\mu\right)_{2}}{\mapsto}\left(s,f',\overline{f}\right)\stackrel{\left(\eta\mu\right)_{1}}{\mapsto}\left(\left(sf'\right)',\overline{sf'},\overline{f}\right)\stackrel{\left(\eta\mu\right)_{2}}{\mapsto}\left(\left(sf'\right)',\left(\overline{sf'}\cdot\overline{f}\right)',\overline{\overline{sf'}\cdot\overline{f}}\right).
\]
Comparing the result to (\ref{eq:NF(sf)}) yields Assumption~(\ref{eq:the_stronger_conditions}),
by Corollary~\ref{cor:immediate-consequence}(\ref{enu:immediate-consequence-2}).
Indeed, equating the first components of the obtained expressions
gives $\left(sf\right)'=\left(sf'\right)'$. Equating the second and
the third components gives $\overline{sf}=\overline{sf'}\cdot\overline{f}$.
\item If $sf'$ is an element of $S_{+}$, then
\begin{equation}
\left(s,f',\overline{f}\right)\stackrel{\left(\eta\mu\right)_{2}}{\mapsto}\left(s,f',\overline{f}\right)\stackrel{\left(\eta\mu\right)_{1}}{\mapsto}\left(\left(sf'\right)',1,\overline{f}\right)\stackrel{\left(\eta\mu\right)_{2}}{\mapsto}\left(\left(sf'\right)',\overline{f},1\right).\label{eq:base-case}
\end{equation}
Again, equating the first components of the result and (\ref{eq:NF(sf)})
gives $\left(sf\right)'=\left(sf'\right)'$. Equating the second and
the third components now gives $\overline{sf}=\overline{f}$, but
note that, in the present case, this equality is equivalent to $\overline{sf}=\overline{sf'}\cdot\overline{f}$.
\end{casenv}
Finally, observe that the case $sf'=1$ cannot occur. Namely, if $sf'$
were equal to $1$, then applying $\left(\eta\mu\right)_{1}$ after
(\ref{eq:base-case}) would yield $\left(\overline{f},1,1\right)$
since $\overline{f}$ is in $S_{+}$, and that would contradict Condition~(\ref{eq:stable_212_on_S^3}).
\end{proof}

Lemma~\ref{lem:stable-212-on-S^3-implies-stronger-conditions-for-=00005Cell(g)=00003D2}
suggests itself as an induction base case, which we are going to achieve
in Proposition~\ref{prop:induction}. First we introduce some notation,
in order to simplify exposition.

\begin{notation*}If $\left(r_{1},r_{2},\ldots,r_{n}\right)$ is
the normal form of $f$ in $M$, then the product $r_{1}r_{2}\cdots r_{n-1}$
in $M$ is denoted by $\underline{f}$.\end{notation*}
\begin{prop}
\label{prop:induction}Let $\left(M,S,\eta\right)$ be a factorable
monoid. If Condition~(\ref{eq:stable_212_on_S^3}) is satisfied for
each $S_{+}$-word $\left(r,s,t\right)$ such that $\varphi_{212}\left(r,s,t\right)$
contains $1$, then Assumption~(\ref{eq:the_stronger_conditions})
is valid for all $s$ in $S_{+}$ and $f$ in $M$.
\end{prop}

\begin{proof}
By Proposition~\ref{prop:stronger-factorable-implies-(4,3)}, the
quadratic normalisation corresponding to the given factorability structure
is of class $\left(4,3\right)$. Then, Remark~\ref{rem:leftmost-letter}
(on Lemma~\ref{lem:normalisation-letter-next-to-normal-word}) implies
the equality $\left(sf\right)'=\left(sf'\right)'$ for all $s$ in
$S_{+}$ and $f$ in $M$.

We need to show that the equality $\overline{sf}=\overline{sf'}\cdot\overline{f}$
also holds for all $s$ in $S_{+}$ and $f$ in $M$. We are going
to achieve this using induction on the length of $f$. Let $P\left(n\right)$
be the statement: the equality $\overline{sf}=\overline{sf'}\cdot\overline{f}$
holds for all $s$ in $S_{+}$ and $f$ of length $n$ in $M$. The
statement $P\left(2\right)$ (resp. $P\left(1\right)$) holds by Lemma~\ref{lem:stable-212-on-S^3-implies-stronger-conditions-for-=00005Cell(g)=00003D2}
(resp. trivially).

Let $n$ be an integer greater than $2$, and suppose that the statement
$P\left(n-1\right)$ holds. Fix an arbitrary $s$ in $S_{+}$ and
$f$ of length $n$ in $M$. Since the length of $\underline{f}$
equals $n-1$, the equality $\overline{s\underline{f}}=\overline{sf'}\cdot\overline{\underline{f}}$
holds by the inductive hypothesis.\footnote{Statement $P\left(1\right)$ could not serve as the base case because
$\overline{\underline{f}}$ in the inductive hypothesis would not
be defined, which is why we needed Lemma~\ref{lem:stable-212-on-S^3-implies-stronger-conditions-for-=00005Cell(g)=00003D2}.} Notice that notation $\overline{\underline{f}}$ is not ambiguous
since $\overline{\left(\underline{f}\right)}$ equals $\underline{\left(\overline{f}\right)}$
due to the fact that the normal form is everywhere stable. Denote
the normal form of $f$ by $\left(r_{1},r_{2},\ldots,r_{n}\right)$.
Multiplying the equality $\overline{s\underline{f}}=\overline{sf'}\cdot\overline{\underline{f}}$
by $r_{n}$ on the right yields
\begin{equation}
\overline{s\underline{f}}\cdot r_{n}=\overline{sf'}\cdot\overline{f}.\label{eq:inductive_hypothesis}
\end{equation}
Hence it suffices to show that the equality $\overline{s\underline{f}}\cdot r_{n}=\overline{sf}$
holds.

Let us compute the normal forms of $\overline{s\underline{f}}$ and
$\overline{sf}$ using Lemma~\ref{lem:normalisation-letter-next-to-normal-word}.
Denoting  $s_{1}\coloneqq s$ and $t_{i}|s_{i}\coloneqq\overline{N}\left(s_{i-1}|r_{i}\right)$
for $i\in\left\{ 2,\ldots,n-1\right\} $, we obtain
\[
\nf\left(\overline{s\underline{f}}\right)=t_{2}|t_{3}|\cdots|t_{n-1}|s_{n-1}
\]
and
\[
\nf\left(\overline{sf}\right)=t_{2}|t_{3}|\cdots|t_{n}|s_{n},
\]
as displayed by the diagram
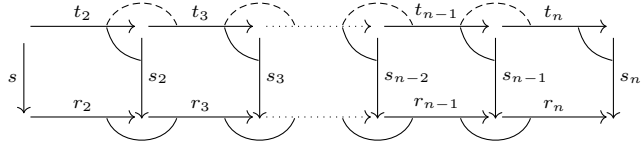
\[
\begin{tikzcd}[row sep=large, column sep=large] \null \arrow[r, shorten <=-2pt, shorten >=-2pt, "t_{2}"] \arrow[r, phantom, ""{name=t2}', near end] \arrow[d, "s"'] & \null \arrow[r, shorten <=-2pt, shorten >=-2pt, "t_{3}"] \arrow[r, phantom, ""{name=t3_1}, near start] \arrow[r, phantom, ""{name=t3_2}', near end] \arrow[d, shorten <=-2pt, shorten >=-2pt, "s_{2}"] \arrow[d, phantom, ""{name=s2}', near start] & \null \arrow[r, shorten <=-2pt, shorten >=-2pt, dotted] \arrow[r, phantom, ""{name=tk_1}, near start] \arrow[r, phantom, ""{name=tk_2}', near end] \arrow[d, shorten <=-2pt, shorten >=-2pt, "s_{3}"] \arrow[d, phantom, ""{name=s3}', near start]  & \null \arrow[r, shorten <=-2pt, shorten >=-2pt, "t_{n-1}"] \arrow[r, phantom, ""{name=tn-1_1}, near start] \arrow[r, phantom, ""{name=tn-1_2}', near end] \arrow[d, shorten <=-2pt, shorten >=-2pt, "s_{n-2}"] \arrow[d, phantom, ""{name=sn-2}', near start] & \null \arrow[r, shorten <=-2pt, shorten >=-2pt, "t_{n}"] \arrow[r, phantom, ""{name=tn_1}, near start] \arrow[r, phantom, ""{name=tn_2}', near end] \arrow[d, shorten <=-2pt, shorten >=-2pt, "s_{n-1}"] \arrow[d, phantom, ""{name=sn-1}', near start] & \null \arrow[d, shorten <=-2pt, shorten >=-2pt, "s_{n}"] \arrow[d, phantom, ""{name=sn}', near start] \\ \null \arrow[r, shorten <=-2pt, shorten >=-2pt, "r_{2}"] \arrow[r, phantom, ""{name=r2}', near end] & \null \arrow[r, shorten <=-2pt, shorten >=-2pt, "r_{3}"] \arrow[r, phantom, ""{name=r3_1}', near start] \arrow[r, phantom, ""{name=r3_2}', near end] & \null \arrow[r, shorten <=-2pt, shorten >=-2pt, dotted] \arrow[r, phantom, ""{name=rk_1}', near start] \arrow[r, phantom, ""{name=rk_2}', near end] & \null \arrow[r, shorten <=-2pt, shorten >=-2pt, "r_{n-1}"] \arrow[r, phantom, ""{name=rn-1_1}', near start] \arrow[r, phantom, ""{name=rn-1_2}', near end] & \null \arrow[r, shorten <=-2pt, shorten >=-2pt, "r_{n}"] \arrow[r, phantom, ""{name=rn_1}', near start] & \null \arrow[from=t2, to=s2, shorten <=-2pt, shorten >=-2pt, bend right, dash] \arrow[from=t2, to=t3_1, shorten <=-2pt, shorten >=-2pt, bend left=70, dash, dashed] \arrow[from=r2, to=r3_1, shorten <=-2pt, shorten >=-2pt, bend right=70, dash] \arrow[from=t3_2, to=s3, shorten <=-2pt, shorten >=-2pt, bend right, dash] \arrow[from=t3_2, to=tk_1, shorten <=-2pt, shorten >=-2pt, bend left=70, dash, dashed] \arrow[from=r3_2, to=rk_1, shorten <=-2pt, shorten >=-2pt, bend right=70, dash] \arrow[from=tk_2, to=sn-2, shorten <=-2pt, shorten >=-2pt, bend right, dash] \arrow[from=tk_2, to=tn-1_1, shorten <=-2pt, shorten >=-2pt, bend left=70, dash, dashed] \arrow[from=rk_2, to=rn-1_1, shorten <=-2pt, shorten >=-2pt, bend right=70, dash] \arrow[from=tn-1_2, to=sn-1, shorten <=-2pt, shorten >=-2pt, bend right, dash] \arrow[from=tn-1_2, to=tn_1, shorten <=-2pt, shorten >=-2pt, bend left=70, dash, dashed] \arrow[from=rn-1_2, to=rn_1, shorten <=-2pt, shorten >=-2pt, bend right=70, dash] \arrow[from=tn_2, to=sn, shorten <=-2pt, shorten >=-2pt, bend right, dash] \end{tikzcd}
\]with arcs having the same meaning as in the diagram  (\ref{eq:domino}).
Therefore, $\overline{s\underline{f}}=t_{2}t_{3}\cdots t_{n-1}s_{n-1}$.
Multiplying by $r_{n}$ on the right yields
\[
\overline{s\underline{f}}\cdot r_{n}=t_{2}t_{3}\cdots t_{n-1}s_{n-1}r_{n}=t_{2}t_{3}\cdots t_{n-1}t_{n}s_{n}=\overline{sf},
\]
which, together with the equality~(\ref{eq:inductive_hypothesis}),
implies $P\left(n\right)$.
\end{proof}
The results of the present subsection enable the following characterisation. 
\begin{prop}
\label{prop:factorable-(4,3)}Let $\left(M,S,\eta\right)$ be a factorable
monoid. Then the following properties are equivalent.
\begin{enumerate}
\item \label{enu:summary-1}For all $s$ in $S_{+}$ and $f$ in $M$, the
equalities $\left(sf\right)'=\left(sf'\right)'$ and $\overline{s}f=\overline{sf'}\cdot\overline{f}$
hold.
\item \label{enu:summary-2}For all $\left(f,g,h\right)$ in $M^{3}$, the
equality 
\[
\left(\eta\mu\right)_{2121}\left(f,g,h\right)=\left(\eta\mu\right)_{212}\left(f,g,h\right)
\]
 holds.
\item \label{enu:summary-3}For each $S_{+}$-word $\left(r,s,t\right)$
such that $\varphi_{212}\left(r,s,t\right)$ contains $1$, the equality
\[
\left(\eta\mu\right)_{2121}\left(r,s,t\right)=\left(\eta\mu\right)_{212}\left(r,s,t\right)
\]
 holds.
\item \label{enu:summary-4}The quadratic normalisation $\left(S,N_{\varphi}'\right)$
mod $1$ for $M$ is of class $\left(4,3\right)$. 
\end{enumerate}
\end{prop}

\begin{proof}
It has already been known (Lemma~\ref{lem:stronger-conditions-imply-stable-212-on-M^3})
that (\ref{enu:summary-1}) implies (\ref{enu:summary-2}), which,
in turn, clearly implies (\ref{enu:summary-3}). The properties~(\ref{enu:summary-3})
and (\ref{enu:summary-4}) are equivalent by the definitions of the
notions concerned. Finally, Proposition~\ref{prop:induction} says
that (\ref{enu:summary-3}) implies (\ref{enu:summary-1}).
\end{proof}
Let us observe that Proposition~\ref{prop:factorable-(4,3)} can
also be read another way, thanks to Theorem~\ref{thm:characterisation},
as a characterisation of monoids admitting a quadratic normalisation
of class $\left(4,3\right)$ among factorable monoids.
\begin{cor}
A monoid admits a quadratic normalisation of class $\left(4,3\right)$
if, and only if, it admits a factorability structure having any of
the properties (\ref{enu:summary-1}), (\ref{enu:summary-2}) and
(\ref{enu:summary-3}) of Proposition~\ref{prop:factorable-(4,3)}.
\end{cor}

\end{document}